\documentclass[11pt]{article}
\usepackage{amsmath,amsfonts,amsthm,amssymb,mathrsfs,amsthm,fullpage}
\usepackage[dvips]{hyperref}
\allowdisplaybreaks
\makeatletter

\@addtoreset{equation}{section}
\makeatother

\newtheorem{theo}{Theorem.}[section]

\newtheorem{lemm}[theo]{Lemma.}

\title{Kernel identities for van Diejen's $q$-difference operators and 
transformation formulas for multiple basic hypergeometric series }
\author{Yasuho Masuda \footnote{Email address: myasuho@math.kobe-u.ac.jp}
}
\date{Department of Mathematics, Kobe University}
\begin{document}
\maketitle
\begin{abstract}
In this paper, we show that the kernel function of
Cauchy type for type $BC$
intertwines the commuting family of van Diejen's $q$-difference operators.
This result gives rise to a transformation formula for
certain multiple basic hypergeometric series of type $BC$. 
We also construct a new infinite family of commuting $q$-difference operators for which the Koornwinder polynomials
are joint eigenfunctions.
\end{abstract}
\tableofcontents

\section{Introduction}
\label{intro}
In the theory of Macdonald polynomials of type $A$,
the kernel function of Cauchy type has been used to derive various important properties of 
Macdonald polynomials \cite{KiN,Ma,Mi}.
Kajihara's Euler transformation formula for multiple basic hypergeometric series can also be
regarded as an application of the kernel function of Cauchy type \cite{Ka}.

Recently, Y. Komori, M. Noumi and J. Shiraishi in \cite{KNS} 
introduced the kernel function $\Phi(x;y |q, t)$ of type $BC$ in the variables $x = (x_1, \ldots, x_m)$ and $y=(y_1, \ldots, y_n)$
relevant to Koornwinder polynomials.
The kernel function $\Phi(x;y |q, t)$
satisfies the following $q$-difference equation:
\begin{equation}
\langle t \rangle D_1^x \Phi(x;y| q, t) - \langle t \rangle \widetilde{D}_1^y \Phi(x;y| q, t)
= \langle t^m \rangle \langle t^{-n} \rangle \langle \alpha^2 t^{m-n-1} \rangle \Phi(x; y | q, t) ,
\label{eq:4.1 of KNS}
\end{equation}
where $\alpha = \sqrt{abcd/q}$ and $\langle z \rangle$ is the multiplicative notation for trigonometric function
\begin{equation}
\langle z \rangle = z^{\frac{1}{2}} - z^{-\frac{1}{2}}=-z^{-\frac{1}{2}}(1-z).
\end{equation}
In this identity, $D^x_1$ is the Koornwinder $q$-difference operator in the $x$ variables 
\begin{align}
&D^x_1 = \sum_{i=1}^m A_i (x) (T_{q, x_i} -1) + 
\sum_{i=1}^m A_i (x^{-1}) (T^{-1}_{q, x_i} -1) , \\
&A_i(x) = \frac{\langle a x_i \rangle \langle bx_i \rangle \langle c x_i \rangle \langle d x_i \rangle}
{\langle x_i^2 \rangle \langle qx_i^2 \rangle}
\prod_{\substack{1 \le j \le m \\ j \neq i}} \frac{\langle t x_i x_j \rangle \langle t x_i x_j^{-1} \rangle}
{\langle x_i x_j \rangle \langle  x_i x_j^{-1}\rangle}, \\
&T_{q, x_i}  f(x_1, \ldots, x_i, \ldots, x_m) = f(x_1, \ldots, q x_i, \ldots, x_m),
\end{align}
and $\widetilde{D}^y_1$ denotes the Koornwinder operator in the 
$y$ variables with the parameters $(a, b, c, d)$ replaced by
$(\sqrt{tq}/a, \sqrt{tq}/b, \sqrt{tq}/c, \sqrt{tq}/d)$.
In this paper, we show that 
$\Phi(x; y|q, t)$ intertwines the whole commuting family of 
van Diejen's $q$-difference operators, which includes the Koornwinder operator as the first member.

In Subsection \ref{subsec:van Diejen}, 
we recall some basic facts on van Diejen's $q$-difference operators.
We also state our main result in Subsection \ref{subsec:Main result}
and prove it in Subsection \ref{subsec:Proof of main result}.
In the proof of the main result, we show a rational function identity 
of $x$ variables and $y$ variables.
By using a method of principal specialization,
from this identity we derive two types of transformation formulas for multiple $q$-series (Theorem \ref{Theo:trans form BC} and Theorem \ref{Theo:trans form C}) in Section \ref{sec:Transformation}.
Theorem \ref{Theo:trans form C} recovers one of the $C$ type
transformation formulas, due to H. Rosengren \cite{R1}.
In Section \ref{sec:New family},
from the special case of Theorem \ref{Theo:trans form BC}, 
we construct a family of explicit $q$-difference operators of 
``row type'' for which the Koornwinder polynomials are the eigenfunctions.

Throughout the present paper, we assume that the base
$q$ is a complex number such that $0 < |q| < 1$.
We also assume that $a, b, c, d, q, t$ are generic complex numbers. 
We use the notation of fractional powers
$a^{\kappa}b^\lambda\cdots x^\mu \cdots$
of multiplicative variables $a, b, \ldots x, \ldots$
as an conventional notation 
$\exp (\kappa A + \lambda B + \cdots \mu x + \cdots)$
under a parametrization
$a=\exp(A),  b=\exp(B),  \ldots,  x= \exp(X),  \ldots$
by additive variables $A, B, \ldots, X, \ldots$.

\section{Kernel identity of Cauchy type}
\label{sec:Kernel identity}
\subsection{Van Diejen's $q$-difference operators}
\label{subsec:van Diejen}
In this subsection, we recall some basic properties of the family of van Diejen's $q$-difference operators.
For further details, we refer the reader to \cite{Diejen,Diejen1,Diejen2,KNS,Ko}.

The family of van Diejen's $q$-difference operators $\{D_r^x(a, b, c, d|q, t)\}^m_{r=0}$ in the variables 
$x$ 
is defined as follows:
\begin{align}
&D_r^x :=D_r^x(a, b, c, d|q, t) = \sum_{\substack{I \subset \{ 1, \ldots , m\} \\
0 \le |I| \le r \\
\epsilon_i = \pm1 (i \in I)}}  V_{\epsilon I, I^c} (x) U_{I^c, r-|I|} (x)T_{q, x}^{(I, \epsilon)} , \\
&V_{\epsilon I, J}(x) =   \prod_{i \in I} 
\frac{\langle a x_i^{\epsilon_i}, b x_i^{\epsilon_i}, c x_i^{\epsilon_i}, d x_i^{\epsilon_i} \rangle}
{\langle x_i^{2\epsilon_i}, qx_i^{2\epsilon_i} \rangle} 
\prod_{\substack{i, j \in I\\ i < j}} 
\frac{\langle t x_i^{\epsilon_i} x_{j}^{\epsilon_{j}}, tq x_i^{\epsilon_i} x_{j}^{\epsilon_{j}} \rangle}
{\langle x_i^{\epsilon_i}x_{j}^{\epsilon_{j}}, qx_i^{\epsilon_i}x_{j}^{\epsilon_{j}} \rangle}
\prod_{\substack{i \in I\\ j \in J}}
\frac{\langle t x_i^{\epsilon_i} x_{j}^{\pm 1}\rangle}{\langle x_i^{\epsilon_i}x_{j}^{\pm 1} \rangle}, \\
&U_{J, r}(x) = \sum_{\substack{ I \subset J \\ |I|=r \\ \delta_i = \pm1 (i \in I)}} (-1)^r
\prod_{i \in I} \frac{\langle a x_i^{\delta_i}, b x_i^{\delta_i}, c x_i^{\delta_i}, d x_i^{\delta_i} \rangle}
{\langle x_i^{2\delta_i}, qx_i^{2\delta_i} \rangle} 
\prod_{\substack{i, j \in I\\ i < j}}  
\frac{\langle t x_i^{\delta_i} x_{j}^{\delta_{j}}, q x_i^{\delta_i} x_{j}^{\delta_{j}}/t \rangle}
{\langle x_i^{\delta_i}x_{j}^{\delta_{j}}, qx_i^{\delta_i}x_{j}^{\delta_{j}} \rangle}
\prod_{\substack{i \in I\\ j \in J\backslash I}} 
\frac{\langle t x_i^{\delta_i} x_{j}^{\pm 1}\rangle}{\langle x_i^{\delta_i}x_{j}^{\pm 1} \rangle}, 
\end{align}
where $I^c = \{1, \ldots, m\} \setminus I$ and $T_{q, x}^{(I, \epsilon)} = \prod_{i \in I} T_{q, x_i}^{\epsilon_i}$.
We also used the shorthand notation
\begin{align}
\langle z_1, \ldots, z_k \rangle  = \langle z_1 \rangle  \cdots \langle z_k \rangle, \quad 
\langle zw^{\pm 1} \rangle  = \langle zw, zw^{-1} \rangle. 
\end{align}
We will use the following notation of $q$-shifted factorial in this paper:
\begin{align}
\langle z \rangle_{q, l}=
\prod_{i=1}^l \langle q^{i-1} z \rangle = (-1)^l q^{-\frac{1}{2} 
\binom{l}{2}}z^{-\frac{l}{2}} (z; q)_l \quad (l=0, 1, 2, \ldots),
\end{align}
where $(z;q)_l = \prod_{i=1}^l (1-q^{i-1}z)$.
For these two types of $q$-shifted factorials, we use the shorthand notation as
\begin{align}
\langle z_1, \ldots, z_k \rangle_{q, l}  = 
\prod_{1 \le i \le k}\langle z_i \rangle_{q, l}, \quad 
\langle zw^{\pm 1} \rangle_{q, l}  = \langle zw\rangle_{q, l}
\langle zw^{-1}\rangle_{q, l}, \\ 
(z_1, \ldots, z_k ; q)_{l}  = \prod_{1 \le i \le k}(z_i; q)_{l}, \quad 
(zw^{\pm 1};q)_{l}  = (zw;q)_l (zw^{-1};q)_l.  
\end{align}
Let $w(z)$ and $v(z)$ denote the following rational functions, respectively:
\begin{align}
w(z) = \frac{\langle az, bz, cz, dz \rangle}{\langle z^2, qz^2 \rangle}, \quad
v(z) = \frac{\langle tz \rangle}{\langle z \rangle}.
\end{align}
Then $V_{\epsilon I, J}(x), U_{J, r}(x)$ are also expressed as
\begin{align}
V_{\epsilon I, J}(x) &=   \prod_{i \in I} w(x_i^{\epsilon_i})
\prod_{\substack{i, j \in I\\ i < j}} v( x_{i}^{\epsilon_i} x_{j}^{\epsilon_{j}} ) v( qx_{i}^{\epsilon_i} x_{j}^{\epsilon_{j}} )
\prod_{\substack{i \in I\\ j \in J}} v( x_{i}^{\epsilon_i} x_{j}^{\pm1}), \\
U_{J, r}(x) &= \sum_{\substack{ I \subset J \\ |I|=r \\ \delta_i = \pm1 (i \in I)}} (-1)^r
\prod_{i \in I} w(x_i^{\delta_i})
\prod_{\substack{i, j \in I\\ i < j}}  
v( x_{i}^{\delta_i} x_{j}^{\delta_{j}} ) v( q^{-1} x_{i}^{-\delta_i} x_{j}^{-\delta_{j}} )
\prod_{\substack{i \in I\\ j \in J\backslash I}}
v( x_{i}^{\delta_i} x_{j}^{\pm1}),
\end{align}
where $v(x_i x_j^{\pm1})$ means $v(x_i x_j) \cdot v(x_i x_j^{-1})$.

Let $W_m$ be the Weyl group of type $BC_m$
acting on the Laurent polynomials in the variables $x=(x_1, \ldots, x_m)$
through the permutations of the indices and the inversions of the variables.
Under the assumption that $a, b, c, d, q, t$ are generic,
for each partition $\lambda=(\lambda_1, \ldots, \lambda_m)$
there exists a unique $W_m$-invariant Laurent polynomial 
$P_\lambda(x) = P_\lambda(x;a,b,c,d|q,t)$, called the Koornwinder polynomial attached to $\lambda$,
satisfying the following conditions \cite{Diejen1}.

(1)\ $P_\lambda(x)$ is expanded by the orbit sums $m_\mu(x) =  \sum_{\nu \in W. \mu} x^{\nu}$ as
\begin{equation}
P_\lambda (x) = m_\lambda(x) + \sum_{\mu < \lambda} c_{\lambda \mu} m_{\mu} (x),
\end{equation}
where $c_{\lambda \mu} \in \mathbb{C}$ and $<$ means the dominance ordering of the partitions.

(2)\ $P_\lambda(x)$ is a joint eigenfunction of van Diejen's $q$-difference operators $D_r^x$:
\begin{equation}
D^x_r P_\lambda(x) = P_\lambda(x) 
 e_r(\alpha  t^{\rho_m} q^\lambda ; \alpha|t),
\end{equation}
where $\rho_m =(m-1, \ldots, 1, 0)$ and $e_r(x; \alpha|t)$ are the interpolation polynomials of column type
defined by
\begin{align}
e_r(x; \alpha| t) &= \sum_{1 \le i_1 < \cdots < i_r \le m} e(x_{i_1} ; t^{i_1 -1} \alpha)
e(x_{i_2} ; t^{i_2 -2} \alpha) \cdots e(x_{i_r} ; t^{i_r-r} \alpha) \\
&= \sum_{1 \le i_1 < \cdots < i_r \le m} e(x_{i_1} ; t^{m-i_1 -r+1} \alpha)  
e(x_{i_2} ; t^{m-i_2 -r+2} \alpha) \cdots e(x_{i_r} ; t^{m-i_r} \alpha) , \\
e(z; w) &= \langle zw \rangle \langle zw^{-1} \rangle = z +z^{-1} -w-w^{-1}.
\end{align}
Note that $e_r(x; \alpha|t)$ is $W_m$-invariant
and satisfies the following interpolation property (See \cite{KNS}):
For any partition $\mu \not\supset (1^r) $,
\begin{equation}
e_r(\alpha t^{\rho_m} q^{\mu} ; \alpha|t) = 0.
\end{equation}

\subsection{Main result}
\label{subsec:Main result}
We recall the definition of the kernel function $\Phi(x; y|q, t)$ of Cauchy type associated with the root systems of type $BC$ 
in the variables $x= (x_1, \ldots, x_m)$ and $y=(y_1, \ldots,y_n)$.
The kernel function $\Phi(x;y |q, t)$ is defined as a
solution of the following linear $q$-difference equations:
\begin{align}
T_{q, x_i} \Phi(x; y|q,t) &=  \Phi(x;y|q, t) 
\prod_{\substack{1 \le l \le n}} 
\frac{e(\sqrt{q/t}x_i; y_l)}
{e(\sqrt{tq}x_i; y_l)} \quad (1 \le i \le m),\\
T_{q, y_k} \Phi(x; y|q, t) &= \Phi(x;y|q, t) 
\prod_{\substack{1 \le j \le m}} 
\frac{e(\sqrt{q/t}y_k; x_j)}
{e(\sqrt{tq} y_k; x_j)}\quad (1 \le k \le n).
\end{align}
Such a $\Phi(x; y|q, t)$ is a multiple of the function
\begin{equation}\label{eq:def of kernel}
\Phi_0(x;y|q, t) =(x_1 \cdots x_m)^{n \gamma}
 \prod_{\substack{1 \le i \le m \\ 1 \le k \le n}} 
 \frac{(\sqrt{tq}x_i y_k^{\pm1}; q)_{\infty}}
 {(\sqrt{q/t}x_i y_k^{\pm1} ; q)_{\infty}},
\end{equation}
by a $q$-periodic function with respect to all the variables $x$ and $y$.
Here $(z; q)_{\infty} = \prod_{i=0}^\infty (1-q^{i}z)$ and
$\gamma$ is a complex number such that $q^\gamma = t$.
We note that four types of 
explicit formulas for kernel function of Cauchy type including $\Phi_0(x;y|q, t)$ are introduced in \cite{KNS}.

For any integer $l$, let $e(z; w)_{q, l}$ be the $q$-shifted factorial of type $BC$
with base point $w$ defined by
\begin{align}
e(z; w)_{q, l} =
\begin{cases} 
e(z; w) e(z; qw) \cdots e(z; q^{l-1}w) \quad (l \ge 0), \\
\displaystyle
\frac{1}{e(z; q^{l}w) e(z; q^{l+1} w) \cdots e(z; q^{-1}w)} \quad  (l < 0).
\end{cases}
\end{align}
We also define a generating function of $D_r^{x}$ and that of $\widetilde{D}_r^{y}$ by
\begin{align}
\mathcal{D}^x(u):=
\mathcal{D}^x(u;a, b,c,d |q, t) = \sum_{r=0}^m (-1)^r 
D_r^x e( u; \alpha)_{t, m-r}, \\
\widetilde{\mathcal{D}}^y(u):
=\mathcal{D}^y(u; \widetilde{a}, \widetilde{b},
\widetilde{c},\widetilde{d}|q, t) 
= \sum_{r=0}^n (-1)^r 
\widetilde{D}_r^y e(u; \widetilde{\alpha})_{t, n-r},
\end{align}
where $(\widetilde{a}, \widetilde{b},\widetilde{c},\widetilde{d})
=(\sqrt{tq}/a, \sqrt{tq}/b, \sqrt{tq}/c, \sqrt{tq}/d)$,
so $\widetilde{\alpha}=t/\alpha$.
We also denoted $D_r^y(\widetilde{a}, \widetilde{b},\widetilde{c},\widetilde{d} |q, t)$
by $\widetilde{D}_r^y$.
For any function $f(z) =f(z;a,b,c,d)$ depending on the parameters $(a, b, c, d)$,
we write $\widetilde{f}(z)=f(z;\widetilde{a}, \widetilde{b},\widetilde{c},\widetilde{d})$.
Then we have the following theorem.
\begin{theo}\label{Theo:Main}
The kernel function $\Phi(x;y |q, t)$ intertwines the $q$-difference operator $\mathcal{D}^x (u)$ 
in the $x$ variables with the $q$-difference operator
$\widetilde{\mathcal{D}}^y (u)$ in the $y$ variables:
\begin{equation}\label{eq:theo1}
\mathcal{D}^x (u) \Phi(x; y|q, t) 
= e(u; \alpha)_{t, m-n} \widetilde{\mathcal{D}}^y (u)  \Phi(x;y|q, t).
\end{equation}
We call this equation a {\em kernel identity of Cauchy type}.
\end{theo}

A proof of this theorem will be given in the next subsection.
We now give some remarks related to Theorem \ref{Theo:Main}.
Firstly, it is known that Theorem \ref{Theo:Main} in the case of $n=0$ holds.
Namely, the constant function $1$ is the eigenfunction of van Diejen's $q$-difference operators \cite{Diejen1}:
\begin{equation}
\mathcal{D}^x (u)\cdot 1 = e(u; \alpha)_{t, m}. \label{eq:eigen 1}
\end{equation}
We will use this fact as the starting point of our proof.
It is also known by \cite{KNS} that
\begin{align}
\sum_{r=0}^m (-1)^r e_r(x;\alpha|t) e(u; \alpha)_{t, m-r} = \prod_{i=1}^m e(u; x_i).
\end{align}
In general, for any partition $\lambda$ we have
\begin{align}
\mathcal{D}^x(u) P_\lambda(x) = P_\lambda(x) \prod_{i=1}^m e(u; \alpha t^{m-i} q^{\lambda_i}).
\end{align}

Secondly,
comparing the coefficient of $e(u; \alpha)_{t, m-1}$ in the left-hand side of (\ref{eq:theo1})
with that in the right-hand side, we obtain (\ref{eq:4.1 of KNS}).
In fact, the $q$-$\text{Saalsch}\ddot{\text{u}}\text{tz}$ sum gives the transformation formula for
the base points of the $q$-shifted factorials of type $BC$:
\begin{align}
e(w; b)_{t,l} 
&=\sum_{0 \le r \le l} (-1)^{r} 
\begin{bmatrix} 
l  \\ 
r  
\end{bmatrix}_{t}
e(t^{\frac{1}{2}(l-1)}b; t^{\frac{1}{2}(1-l)}/a)_{t, r} e(w; a)_{t, l-r} , \label{eq:Saal}\\
\begin{bmatrix} 
l  \\ 
r  
\end{bmatrix}_{t}
&=(-1)^r \frac{\langle t^{-l} \rangle_{t, r}}{\langle t \rangle_{t, r}} .
\end{align}
It follows from this formula that
\begin{align}
&e(u; \alpha)_{t, m-n} e(u; t/\alpha)_{t, n-k} \nonumber \\
&=\sum_{0 \le l \le n-k} (-1)^{l} 
\begin{bmatrix} 
n-k  \\ 
l 
\end{bmatrix}_{t}
e(t^{\frac{1}{2}(n-k+1)}/\alpha; t^{\frac{1}{2}(1+n-2m+k)}/\alpha)_{t,l} 
e(u; \alpha)_{t, m-k-l}.
\end{align}
Comparing the coefficients of $e(u; \alpha )_{t, m-r}$ in the both sides of (\ref{eq:theo1}),
we have 
\begin{align} \label{eq:theorem3.1 another version}
D_r^x \Phi(x; y|q, t) = \sum_{k= 0}^r  \widetilde{D}_k^y 
\begin{bmatrix} 
n-k  \\ 
r-k
\end{bmatrix}_{t}
e( t^{\frac{1}{2}(n-k+1)}/\alpha; t^{\frac{1}{2}(1+n-2m+k)}/\alpha)_{t,r-k} 
\Phi(x;y|q, t). 
\end{align}
The formula (\ref{eq:theorem3.1 another version}) for $r=1$ recovers the result (\ref{eq:4.1 of KNS}) of \cite{KNS}.

\subsection{Proof of the main result}
\label{subsec:Proof of main result}
It is enough to show the case where $m \ge n \ge 0$.
The identity (\ref{eq:theo1}) is equivalent to
\begin{equation}\label{eq:prf1}
\Phi(x;y|q,t)^{-1} \mathcal{D}^x (u) \Phi(x; y|q,t) 
=\Phi(x;y|q,t)^{-1} \widetilde{\mathcal{D}}^y (u) 
\Phi(x;y|q,t) e(u; \alpha)_{t, m-n} .
\end{equation}
Regarding this as a rational function identity of the variable $y_n$,
we prove it by computing the residues and the limits as $y_n \rightarrow \infty$.

The generating function $\mathcal{D}^x(u)$ is expanded as
\begin{align}
\mathcal{D}^x (u) &= \sum_{\substack{I \subset \{ 1, \ldots , m\} \\
\epsilon_i = \pm1 (i \in I)}} (-1)^{|I|} V_{\epsilon I, I^c} (x) U_{I^c} (u; x) T_{q, x}^{(I, \epsilon)} , \\
U_{J}(u; x) &= \sum_{\substack{ I \subset J \\  \delta_i = \pm1 (i \in I)}}
e( u; \alpha)_{t, |J|-|I|} 
\prod_{i \in I} w(x_i^{\delta_i}) 
\prod_{\substack{i, j \in I\\ i < j}}  
v( x_{i}^{\delta_i} x_{j}^{\delta_{j}} ) v( q^{-1} x_{i}^{-\delta_i} x_{j}^{-\delta_{j}} )
\prod_{\substack{i \in I\\ j \in J\backslash I}}
v( x_{i}^{\delta_i} x_{j}^{\pm 1} ) . \label{eq:def U}
\end{align}
Similarly, we expand $\widetilde{\mathcal{D}}_y(u)$.
We also define the rational function $F(z; w)$ in the variables $z=(z_1, \ldots, z_r)$ and $w=(w_1, \ldots, w_s)$ by
\begin{equation}
F(z; w)=
\prod_{\substack{1 \le i \le r \\ 1 \le k \le s}} \frac{e(\sqrt{q/t}z_i; w_k)}
{e(\sqrt{tq}z_i; w_k)} .
\end{equation}
For any subset $I =\{ i_1, \ldots, i_r \} \subset \{1, \ldots, m\}, |I|=r$ and signs $\epsilon_i = \pm1 (i \in I)$,
we write $x_{I}^{\epsilon} = (x_{i_1}^{\epsilon_{i_1}}, \ldots, x_{i_r}^{\epsilon_{i_r}})$.
Then, (\ref{eq:prf1}) is expressed as
\begin{equation}\label{eq:prf2}
\begin{split}
&\sum_{\substack{I \subset \{ 1, \ldots , m\} \\ \epsilon_i = \pm1 (i \in I)}}(-1)^{|I|}
V_{\epsilon I, I^c} (x) U_{I^c} (u; x) F(x_{I}^\epsilon ; y) \\
&=e(u; \alpha)_{t, m-n} 
\sum_{\substack{K \subset \{ 1, \ldots , n\} \\ \epsilon_k = \pm1 (k \in K)}}
(-1)^{|K|}
\widetilde{V}_{\epsilon K, K^c} (y) \widetilde{U}_{K^c} (u; y)
F(y_{K}^\epsilon ; x) .
\end{split}
\end{equation}
We prove this identity by induction on $n$, starting with (\ref{eq:eigen 1}) of the case $n=0$.
We assume that our identity holds when the number of $y$ variables is less than $n$.

Firstly, we consider the residues of the both sides as meromorphic functions in $y_n$.
The left-hand side of (\ref{eq:prf2}) may have the poles at 
\begin{equation}
\begin{split}
y_n = \sqrt{tq} x_i^{\pm 1} , \quad \frac{1}{\sqrt{tq}} x_i^{\pm1} \ ( 1 \le i \le m).
\end{split}
\end{equation}
On the other hand, there may be the poles at 
\begin{equation}
\begin{split}
y_n = & \sqrt{tq} x_i^{\pm 1},  \quad \frac{1}{\sqrt{tq}} x_i^{\pm1} \ ( 1 \le i \le m), 
\quad \pm 1, \quad \pm q^{1/2}, \\
& \pm q^{-1/2}, \quad  y_k^{\pm 1} , \quad  qy_k^{\pm 1}, \quad q^{-1} y_k^{\pm 1} \ ( 1 \le k \le n)
\end{split}
\end{equation}
in the right-hand side.
However, we can check by direct calculation that 
the points other than $y_n = \sqrt{tq} x_i^{\pm 1},$
$\displaystyle \frac{1}{\sqrt{tq}} x_i^{\pm1}\ (i=1, \ldots, m)$
are apparent singular points.
Since (\ref{eq:prf2}) is invariant under the inversions and permutations for $x$ and $y$,
we have only to analyze the residue at the point $y_n = \sqrt{tq} x_m$.

In the left-hand side, the term indexed by $(I, \epsilon)$ has a pole at $y_n = \sqrt{tq} x_m$
if and only if $m \in I$ and $\epsilon_m =1$.
Note that 
\begin{equation}
F(x_{I'}^\epsilon; \sqrt{tq}x_m) = \prod_{i \in I'}
\frac{\langle qx_m x_i^{\epsilon_i}, tx_m/x_i^{\epsilon_i} \rangle}
{\langle tqx_m x_i^{\epsilon_i}, x_m/x_i^{\epsilon_i} \rangle} 
= \prod_{i \in I'} \frac{v(x_m x_i^{-\epsilon_i})}{v(qx_m x_i^{\epsilon_i})},
\end{equation}
where $I' = I \setminus \{m \} $.
Thus it follows that the residue is equal to
\begin{align}
&-\frac{\sqrt{t} \langle ax_m, bx_m, cx_m, dx_m \rangle(\sqrt{tq}-\sqrt{q/t})x_m}
{\langle x_m^2, tqx_m^2 \rangle} \prod_{1 \le j \le m-1} v(x_j^{\pm 1} x_m) \nonumber \\ 
& \cdot F(x_m; y') \times  (\text{l.h.s. for the case of $(x'; y')$}),
\end{align}
where $x'= (x_1, \ldots, x_{m-1})$ and $y'=(y_1, \ldots, y_{n-1})$.

In the right-hand side,
the term indexed by $(K, \epsilon)$ cannot have a pole at $y_n = \sqrt{tq} x_m$
unless $n \in K$ and $\epsilon_n=-1$.
The corresponding residue is equal to
\begin{align}
&-\frac{\sqrt{t} \langle ax_m, bx_m, cx_m, dx_m \rangle(\sqrt{tq}-\sqrt{q/t})x_m}
{\langle x_m^2, tqx_m^2 \rangle}  
\prod_{1 \le j \le m-1} v(x_j^{\pm 1} x_m) \nonumber \\
& \cdot
F(x_m; y') \times (\text{r.h.s. for the case of $(x';y')$}).
\end{align}
Therefore it follows from the induction hypothesis that 
the residues of the both sides at the point $y_n= \sqrt{tq}x_m$ are equal. 

Next, we calculate the limits of the both sides as $y_n \rightarrow \infty$.
It is easy to check
\begin{equation}
\lim_{y_n \rightarrow \infty} (\text{l.h.s.}) = (\text{l.h.s. for the case 
of $(x;y')$}) .
\end{equation}
We consider the limit of the individual terms of the right-hand side in the following three cases:
\begin{center}
(i)\ $n \in K$ and $\epsilon_{n}=1$,\qquad (ii)\ $n \in K$ and $\epsilon_{n}=-1$, \qquad (iii)\ $n \notin K$.
\end{center}
By direct calculation, we can check in the case (i) and (ii) respectively as follows:
\begin{align}
&\lim_{y_n \rightarrow \infty } \left(\sum_{\substack{n \in K \\ \epsilon_n =1}} 
(-1)^{|K|}\widetilde{V}_{\epsilon K, K^c} (y) \widetilde{U}_{K^c} (u; y)F(y_K^{\epsilon}; x) \right) \nonumber \\ 
&= -\widetilde{\alpha} t^{n-m-1} 
\sum_{\substack{K' \subset \{ 1, \ldots , n-1\} \\ \epsilon_k = \pm1 (k \in K')}}
(-1)^{|K'|}
\widetilde{V}_{\epsilon K', {K'}^c} (y') \widetilde{U}_{K'^c} (u; y')F(y_{K'}^\epsilon; x) , \label{eq:lim of r1} \\
&\lim_{y_n \rightarrow \infty } \left(\sum_{\substack{ n \in K \\ \epsilon_n = - 1\ }}
(-1)^{|K|}\widetilde{V}_{\epsilon K, K^c} (y) \widetilde{U}_{K^c} (u; y)F(y_K^{\epsilon}; x)  \right) \nonumber \\
&= -\widetilde{\alpha}^{-1} t^{m-n+1} 
\sum_{\substack{K' \subset \{ 1, \ldots , n-1\} \\ \epsilon_k = \pm1 (k \in K')}}
(-1)^{|K'|}\widetilde{V}_{\epsilon K', {K'}^c} (y') \widetilde{U}_{K'^c} (u; y') F(y_{K'}^\epsilon; x) \label{eq:lim of r2}.
\end{align}
In the case (iii), 
we divide the sum
\begin{equation}
\widetilde{U}_{K^c}(u; y) = 
\sum_{\substack{ L \subset K^c \\ \delta_k = \pm1 (k \in L)}}
e( u; \widetilde{\alpha})_{t, |K^c|-|L|} \prod_{k \in L} \widetilde{w}(y_k^{\delta_k})
\prod_{\substack{k, l \in L\\ k < l}}  
v( y_{k}^{\delta_k} y_{l}^{\delta_{l}} ) 
v( q^{-1} y_{k}^{-\delta_k} y_{l}^{-\delta_{l}} )
\prod_{\substack{k \in L\\ l \in K^c\backslash L}}
v( y_{k}^{\delta_k} y_{l}^{\pm 1} ) 
\end{equation}
into the three groups of terms as 
\begin{center}
(a)\ $ n \in L$ and $\delta_n={1}$, \qquad
(b)\ $n \in L$ and $\delta_n={-1}$, \qquad
(c)\ $n \notin L$.
\end{center}
Combining the limits of these three cases, we obtain
\begin{align}
&\lim_{y_n \rightarrow \infty } \left( \sum_{\substack{n \notin K}} (-1)^{|K|}
\widetilde{V}_{\epsilon K, K^c} (y) \widetilde{U}_{K^c} (u; y) F(y_K^{\epsilon}; x) \right) \nonumber \\
&= (u+u^{-1})
\sum_{\substack{K' \subset \{ 1, \ldots , n-1\} \\ \epsilon_k = \pm1 (k \in K')}}
(-1)^{|K'|}
\widetilde{V}_{\epsilon K', {K'}^c} (y') \widetilde{U}_{K'^c} (u; y') F(y_{K'}^{\epsilon}; x). \label{eq:lim of r3}
\end{align}
From
(\ref{eq:lim of r1}), (\ref{eq:lim of r2}) and (\ref{eq:lim of r3}),
it follows that
\begin{equation}
\lim_{y_n \rightarrow \infty} (\text{r.h.s.}) 
= (\text{r.h.s. for the case of $(x;y')$}),
\end{equation}
and hence we complete the proof of Theorem \ref{Theo:Main}.

Replacing $(q, t)$ by $(t, q)$, the formula (\ref{eq:prf2}) can be rewritten
explicitly as follows.
\begin{theo}\label{Theo:set ver}
Given two sets of variables $x=(x_1, \ldots, x_m)$ and
$y=(y_1, \ldots, y_n)$, the following identity holds: 
\begin{align}
&\sum_{\substack{I \subset \{ 1, \ldots , m\} \\ \epsilon_i = \pm1 (i \in I)}} 
\Biggl( (-1)^{|I|}
\prod_{i \in I} \frac{\langle ax_i^{\epsilon_i}, bx_i^{\epsilon_i}, cx_i^{\epsilon_i}, dx_i^{\epsilon_i} \rangle}
{\langle x_i^{2\epsilon_i}, tx_i^{2\epsilon_i} \rangle} 
\prod_{\substack{i, j \in I\\ i < j}}
\frac{\langle qx_i^{\epsilon_i} x_{j}^{\epsilon_{j}}, tqx_i^{\epsilon_i} x_{j}^{\epsilon_{j}} \rangle}
{\langle x_i^{\epsilon_i}x_{j}^{\epsilon_{j}}, tx_i^{\epsilon_i}x_{j}^{\epsilon_{j}} \rangle}
\prod_{\substack{i \in I\\ j \in I^c}}
\frac{\langle q x_i^{\epsilon_i} x_{j}^{\pm 1}\rangle}{\langle x_i^{\epsilon_i}x_{j}^{\pm 1} \rangle} \nonumber \\
& \cdot \sum_{\substack{ J \subset I^c \\  \delta_i = \pm1 (i \in J)}} 
\Biggl(e(u; \sqrt{q/t} \alpha)_{q, |I^c|-|J|} 
\prod_{i \in J} 
\frac{\langle a x_i^{\delta_i}, b x_i^{\delta_i}, c x_i^{\delta_i}, d x_i^{\delta_i} \rangle}
{\langle x_i^{2\delta_i}, tx_i^{2\delta_i} \rangle} 
\prod_{\substack{i, j \in J\\ i < j}}  
\frac{\langle q x_i^{\delta_i} x_{j}^{\delta_{j}}, t x_i^{\delta_i} x_{j}^{\delta_{j}}/q\rangle}
{\langle x_i^{\delta_i}x_{j}^{\delta_{j}}, tx_i^{\delta_i}x_{j}^{\delta_{j}}\rangle} \nonumber \\
&\cdot
\prod_{\substack{i \in J\\ j \in I^c\backslash J}} 
\frac{\langle q x_i^{\delta_i} x_{j}^{\pm 1}\rangle}
{\langle x_i^{\delta_i}x_{j}^{\pm 1} \rangle} \Biggr) 
\prod_{\substack{i \in I \\ 1 \le k \le n}} 
\frac{e(\sqrt{t/q}x_i^{\epsilon_i}; y_k)}
{e(\sqrt{tq}x_i^{\epsilon_i}; y_k)}\Biggr)  \nonumber \\
&=e(u; \sqrt{q/t} \alpha)_{q, m-n}
\sum_{\substack{K \subset \{ 1, \ldots , n\} \\ \epsilon_k=\pm 1 (k \in K)}} 
\Biggl( (-1)^{|K|}\prod_{k \in K}
\frac{\langle \sqrt{tq}y_k^{\epsilon_k}/a, \sqrt{tq}y_k^{\epsilon_k}/b ,
\sqrt{tq}y_k^{\epsilon_k}/c, \sqrt{tq}y_k^{\epsilon_k}/d \rangle}
{\langle y_k^{2\epsilon_k}, ty_k^{2\epsilon_k} \rangle} \nonumber \\ 
&\cdot
\prod_{\substack{k, l \in K\\ k < l}} 
\frac{\langle qy_k^{\epsilon_k} y_{l}^{\epsilon_{l}}, tqy_k^{\epsilon_k} y_{l}^{\epsilon_{l}} \rangle}
{\langle y_k^{\epsilon_k}y_{l}^{\epsilon_{l}}, ty_k^{\epsilon_k}y_{l}^{\epsilon_{l}} \rangle}
\prod_{\substack{k \in K\\ l \in K^c}}
\frac{\langle q y_k^{\epsilon_k} y_{l}^{\pm 1}\rangle}{\langle y_k^{\epsilon_k}y_{l}^{\pm 1} \rangle}
\sum_{\substack{ L \subset K^c \\ \delta_k = \pm1 (k \in L)}} 
\Biggl( e(u; \sqrt{tq}/\alpha)_{q, |K^c|-|L|}\nonumber \\
&\cdot
\prod_{k \in L} 
\frac{\langle \sqrt{tq}y_k^{\delta_k}/a, \sqrt{tq}y_k^{\delta_k}/b ,
\sqrt{tq}y_k^{\delta_k}/c, \sqrt{tq}y_k^{\delta_k}/d \rangle}
{\langle y_k^{2\delta_k}, ty_k^{2\delta_k} \rangle}
\prod_{\substack{k, l \in L\\ k < l}} 
\frac{\langle q y_k^{\delta_k} y_{l}^{\delta_{l}}, t y_k^{\delta_k} y_{l}^{\delta_{l}}/q \rangle}
{\langle y_k^{\delta_k}y_{l}^{\delta_{l}}, ty_k^{\delta_k}y_{l}^{\delta_{l}}\rangle} \nonumber \\
&\cdot
\prod_{\substack{k \in L\\ l \in K^c\backslash L}} 
\frac{\langle q y_k^{\delta_k} y_{l}^{\pm 1}\rangle}{\langle y_k^{\delta_k}y_{l}^{\pm 1} \rangle}\Biggr)
\prod_{\substack{k \in K \\ 1 \le i \le m}} 
\frac{e(\sqrt{t/q}y_k^{\epsilon_k}; x_i)}
{e(\sqrt{tq}y_k^{\epsilon_k}; x_i)}
 \Biggr) 
\label{eq:main theo q}.
\end{align}
\end{theo}

\section{Transformation formulas for multiple basic hypergeometric series}
\label{sec:Transformation}
In the case of type $A$,
the kernel function of Cauchy type intertwines  
the Macdonald $q$-difference operators \cite{MiN}.
This property gives the rational function identity which is similar to 
(\ref{eq:main theo q}).
Applying certain specializations to this identity,
Kajihara \cite{Ka} derived the Euler transformation formula for multiple basic hypergeometric series.
In the same way, we propose two types of transformation formulas for multiple basic hypergeometric series.

\subsection{Type $BC$ case}
\label{subsec:type BC}
In this subsection, we derive a transformation formula of type $BC$
from Theorem \ref{Theo:set ver}.
We take the multi-indices $\alpha=(\alpha_1, \ldots, \alpha_M)
\in \mathbb{N}^M$ and $\beta
=(\beta_1, \ldots, \beta_N) \in \mathbb{N}^N$
such that 
$|\alpha| :=\sum_{i=1}^M \alpha_i = m, |\beta|=\sum_{k=1}^N \beta_k =n$.
Here $\mathbb{N}$ is the set of non-negative integers.
Then we consider the following specializations:
\begin{align}
x&=p_{\alpha}(z;q) := (z_1, qz_1, \ldots, q^{\alpha_1-1} z_1;
z_2, qz_2, \ldots, q^{\alpha_2-1} z_2; \ldots ;z_M, qz_M, \ldots, q^{\alpha_M-1} z_M) , \\
y&=p_{\beta}(w;q) := (w_1, qw_1, \ldots, q^{\beta_1-1} w_1;
w_2, qw_2, \ldots, q^{\beta_2-1} w_2; \ldots ;w_N, qw_N, \ldots, q^{\beta_N-1} w_N).
\end{align}
These specializations are called multiple principal specializations. 
We apply these to (\ref{eq:main theo q}).

For each pair $(I, \epsilon)$,
we divide the subset $I$ as $I^+ \sqcup I^-$ by setting
\begin{align}
I^+=\{ i  \in I| \epsilon_i =1 \} ,\quad I^-=\{i \in I|\epsilon_i=-1 \}.
\end{align}
Similarly, we divide the subset $J$ by $J^+=\{i \in J|\delta_i=1 \}$
and $J^- = \{i \in J| \delta_i=-1\}$ for each pair $(J, \delta)$.
We also denote the complement $\{1, \ldots, m\} \backslash (I \cup J)$ by $C$.
Using these symbols, we rewrite the left-hand side of (\ref{eq:main theo q}) as
\begin{align}
&\sum_{\substack{I^+\sqcup I^- \sqcup J^+\sqcup J^- \sqcup C \\= \{1, \ldots, m \}}}
\Biggl( (-1)^{|I^+| +|I^-|}e(u; \sqrt{q/t} \alpha)_{q, |C|}
 \prod_{i \in I^+ \sqcup J^+}\frac{\langle ax_i, bx_i, cx_i, dx_i\rangle}
{\langle x_i^2, tx_i^2 \rangle} 
\nonumber \\
&\cdot \prod_{i \in I^- \sqcup J^-} \frac{\langle a/x_i, b/x_i, c/x_i, d/x_i\rangle}
{\langle x_i^{-2}, tx_i^{-2} \rangle} \prod_{\substack{i, j \in I^+\\ i < j}} 
\frac{\langle q x_i x_{j}\rangle_{t, 2}}
{\langle x_i x_{j}\rangle_{t, 2}}
\prod_{\substack{i, j \in I^-\\ i < j}} 
\frac{\langle q x_i^{-1} x_{j}^{-1}\rangle_{t, 2}}
{\langle x_i^{-1} x_{j}^{-1}\rangle_{t, 2}}
\prod_{\substack{i \in I^+\\ j \in I^-}} 
\frac{\langle q x_i x_{j}^{-1}\rangle_{t, 2}}
{\langle x_i x_{j}^{-1}\rangle_{t, 2}}\nonumber \\
&\cdot 
\prod_{\substack{i \in I^+\\ j \in J^- \sqcup C \sqcup J^+}} 
\frac{\langle q x_i x_{j}^{\pm 1}\rangle}
{\langle x_i x_{j}^{\pm 1}\rangle}
\prod_{\substack{i \in I^-\\ j \in J^- \sqcup C \sqcup J^+}} 
\frac{\langle q x_i^{-1} x_{j}^{\pm 1}\rangle}
{\langle x_i^{-1} x_{j}^{\pm 1}\rangle}
\prod_{\substack{i, j \in J^+\\ i < j}} 
\frac{\langle q x_i x_{j}, tx_i x_j/q \rangle}
{\langle x_i x_j\rangle_{t, 2}}  
\prod_{\substack{i, j \in J^-\\ i < j}} 
\frac{\langle q x_i^{-1} x_{j}^{-1}, t x_i^{-1} x_{j}^{-1}/q \rangle}
{\langle x_i^{-1} x_{j}^{-1}\rangle_{t, 2}}
\nonumber \\
&\cdot \prod_{\substack{i \in J^+\\ j \in J^-}} 
\frac{\langle q x_i x_{j}^{-1}, t x_i x_{j}^{-1}/q \rangle}
{\langle x_i x_{j}^{-1}\rangle_{t, 2}} 
\prod_{\substack{i \in J^+\\ j \in C }} 
\frac{\langle q x_i x_{j}^{\pm 1}\rangle}
{\langle x_i x_{j}^{\pm 1}\rangle}
\prod_{\substack{i \in J^-\\ j \in  C }} 
\frac{\langle q x_i^{-1} x_{j}^{\pm 1}\rangle}
{\langle x_i^{-1} x_{j}^{\pm 1}\rangle} 
%\nonumber \\
%&\cdot
\prod_{\substack{i \in I^+ \\ 1 \le k \le n}} 
\frac{\langle \sqrt{t/q} x_i y_{k}^{\pm 1}\rangle}
{\langle \sqrt{tq} x_i y_{k}^{\pm 1}\rangle}
\prod_{\substack{i \in I^- \\ 1 \le k \le n}} 
\frac{\langle \sqrt{t/q}  x_i^{-1} y_{k}^{\pm 1}\rangle}
{\langle \sqrt{tq} x_i^{-1} y_{k}^{\pm 1}\rangle} \Biggr). \label{eq:before p.s.}
\end{align}
We first consider the principal specialization 
$x_i =q^{i-1}z\ (1 \le i \le m)$ of a single block.
Noting that $\langle qx_i/x_{i+1} \rangle=0 \ (i=1, \ldots, m-1)$,
we find that non-zero terms arise from the divisions of the following form:\begin{equation}
\begin{split}
&I^-=\{1, 2, \ldots, i_1 \}, \quad J^- =\{i_1+1, i_1+2, \ldots, i_2 \}, \quad C =\{i_2+1, i_2+2, \ldots, i_3 \},  \\
&J^+=\{i_3+1, i_3+2, \ldots, i_4 \}, \quad I^+ =\{i_4+1, i_4+2, \ldots, m \},
\quad 0 \le i_1 \le i_2 \le i_3 \le i_4 \le m.
\end{split}
\end{equation}

Next, we consider the multiple principal specializations $x=p_{\alpha}(z; q),
y= p_{\beta}(w;q)$.
We replace the index set $\{1, \ldots, m \}$ by
\begin{equation}
\{1, \ldots , m\} = \{(i, a) | 1 \le i \le M, \ 0 \le a < \alpha_i \} \label{eq:replace of index}
\end{equation}
and write $x_{(i, a)} = q^a z_i $.
For any two multi-indices $\mu , \nu \in \mathbb{N}^M$,
if $\mu_i \le \nu_i \ (i= 1, \ldots, M)$ then
we write $\mu \le \nu$.
From the same argument as above applied to each block,
$I^-, J^- , C ,J^+, I^+$
are parametrized by the four multi-indices
$\mu^-, \nu^-, \nu^+, \mu^+ \in \mathbb{N}^M$ such that
$0 \le \mu^- \le \nu^- \le \nu^+ \le \mu^+ \le \alpha$
as follows:
\begin{equation}
\begin{split}
I^- &= \{(i, a) | 1 \le i \le M, \ 0 \le a < \mu_i^- \} ,\\
J^- &= \{(i, a) | 1 \le i \le M, \mu_i^- \le a < \nu_i^- \} ,\\
C &= \{(i, a) | 1 \le i \le M, \ \nu_i^- \le a < \nu_i^+ \} ,\\
J^+ &= \{(i, a) | 1 \le i \le M, \ \nu_i^+ \le a < \mu_i^+ \} ,\\
I^+ &= \{(i, a) | 1 \le i \le M, \ \mu_i^+ \le a < \alpha_i \}. 
\end{split}
\end{equation}
In the following, 
we omit the base $q$ in the $q$-shifted factorials $\langle z \rangle_{q, k}$
and $e(z;w)_{q, k}$.

With this parametrization (\ref{eq:replace of index}) of indices,
the formula
(\ref{eq:before p.s.}) specialized by $x=p_{\alpha}(z; q),
y= p_{\beta}(w;q)$ gives rise to
\begin{align}
&\prod_{1 \le i \le M} \frac{\langle az_i, bz_i, cz_i, dz_i\rangle_{\alpha_i}}
{\langle z_i^2, tz_i^2 \rangle_{\alpha_i}}
\prod_{1 \le i < j \le M}
\frac{\langle q^{\alpha_j}z_i z_j,  tq^{\alpha_j}z_i z_j\rangle_{\alpha_i}}
{\langle z_i z_j,  tz_i z_j\rangle_{\alpha_i}}
\prod_{\substack{1\le i \le M\\ 1 \le k \le N}}
\frac{\langle \sqrt{t/q}z_i w_k,  \sqrt{tq}z_i/q^{\beta_k} w_k \rangle_{\alpha_i}}
{\langle q^{\beta_k} \sqrt{t/q} z_i w_k,  \sqrt{tq} z_i/w_k \rangle_{ \alpha_i}} \nonumber \\
&\cdot \sum_{0 \le \mu^{-} \le \nu^{-} \le \nu^{+} \le \mu^{+} \le \alpha}
\Biggl( (-1)^{|\alpha|+|\nu^+| + |\nu^-|}
e(u; \sqrt{q/t}\alpha)_{|\nu^+|-|\nu^-|} 
\prod_{1 \le i \le M} \frac{\langle z_i/a, z_i/b, z_i/c, z_i/d\rangle_{\nu_i^{-}}}
{\langle az_i, bz_i, cz_i, dz_i\rangle_{\nu_i^{+}}} \nonumber \\
&\cdot
\prod_{1 \le i \le j \le M}
\frac{\langle q^{\mu_i^- + \mu_j^-} z_i z_j/tq, tq^{\mu_i^+ + \mu_j^+} z_i z_j/q\rangle}
{\langle z_i z_j/tq, t z_i z_j/q\rangle}
\prod_{1 \le i < j \le M}
\frac{\langle q^{\mu_i^- - \mu_j^-} z_i/z_j, q^{\mu_i^+ - \mu_j^+} z_i/z_j\rangle}{\langle z_i/z_j, z_i/z_j\rangle} \nonumber \\
&\cdot
\prod_{1 \le i < j \le M}
\frac{\langle q^{\nu_i^- - \nu_j^-} z_i/z_j, q^{\nu_i^+ - \nu_j^+} z_i/z_j\rangle
\langle  z_i z_j/t, z_i z_j/q \rangle_{\nu_i^- + \nu_j^-}}
{\langle z_i/z_j, z_i/z_j\rangle \langle  z_i z_j, tz_i z_j/q \rangle_{\nu_i^+ + \nu_j^+}}
\prod_{1 \le i, j \le M}
\frac{\langle q^{\mu_i^- + \mu_j^+} z_i z_j/q\rangle}{\langle z_i z_j/q\rangle} \nonumber \\
&\cdot
\prod_{1 \le i, j \le M}
\frac{\langle q^{\nu_i^- + \nu_j^+} z_i z_j/q,
q^{\mu_i^- - \mu_j^+} z_i/tz_j \rangle 
\langle  z_i z_j/q \rangle_{\mu_i^- + \nu_j^+}
\langle  tz_i z_j/q \rangle_{\mu_i^+ + \nu_j^+}
\langle  z_i z_j/tq \rangle_{\mu_i^-}
\langle  z_i z_j/q \rangle_{\mu_i^+ }
}
{\langle z_i z_j/q, z_i/tq^{\mu_j^+}z_j\rangle \langle  z_i z_j/t \rangle_{\mu_i^-+ \nu_j^-}
\langle  z_i z_j \rangle_{\mu_i^+ + \nu_j^-}
\langle  q^{\alpha_j}z_i z_j \rangle_{\mu_i^-}
\langle  tq^{\alpha_j} z_i z_j \rangle_{\mu_i^+}
} \nonumber \\
&\cdot
\prod_{1 \le i, j \le M}
\frac{\langle  z_i/q^{\nu_j^-}z_j , z_i/tq^{\alpha_j}z_j \rangle_{\mu_i^-}
\langle  z_i/q^{\alpha_j}z_j \rangle_{\mu_i^+}
\langle  qz_i/tq^{\nu_j^+}z_j , z_i/q^{\nu_j^+}z_j \rangle_{\nu_i^-}
\langle  z_i/q^{\mu_j^+}z_j \rangle_{\nu_i^+}}
{\langle qz_i/tq^{\nu_j^+}z_j, qz_i/ z_j \rangle_{\mu_i^-}
\langle qz_i/ z_j \rangle_{\mu_i^+}
\langle qz_i/tq^{\mu_j^+}z_j, qz_i/ z_j \rangle_{\nu_i^-}
\langle qz_i/ z_j \rangle_{\nu_i^+}} \nonumber \\
&\cdot
\prod_{\substack{1\le i \le M\\ 1 \le k \le N}}
\frac{\langle q^{\beta_k} z_i w_k/\sqrt{tq}, \sqrt{q/t} z_i/w_k \rangle_{\mu_i^-}
\langle q^{\beta_k} \sqrt{t/q} z_i w_k,  \sqrt{tq} z_i/w_k \rangle_{\mu_i^+}}
{\langle z_i w_k/\sqrt{tq}, \sqrt{q/t} z_i/q^{\beta_k} w_k \rangle_{\mu_i^-}
\langle \sqrt{t/q}z_i w_k,   \sqrt{tq}z_i/q^{\beta_k}w_k \rangle_{\mu_i^+}} \Biggr)
\label{eq:after p.s.}.
\end{align}
We used $\langle a \rangle_k$ version of some formulas of Appendix I in \cite{GR}.
We also used the formula
\begin{align}
\prod_{1 \le i \neq j \le m} \frac{\langle qx_i/x_j \rangle_{\lambda_i}}{\langle x_i/q^{\lambda_j} x_j \rangle_{\lambda_i}}
= \prod_{1 \le i < j \le m} \frac{\langle q^{\lambda_i-\lambda_j}x_i/x_j \rangle}{\langle x_i/x_j \rangle},
\end{align}
due to Milne \cite[Lemma 6.11]{Mil}.

We apply the same specializations to the right-hand side of (\ref{eq:main theo q}).
If we denote (\ref{eq:after p.s.}) by \\
$F^{(M, N)}_{\alpha, \beta} (z;w;a,b,c,d)$,
Theorem \ref{Theo:set ver} implies the following duality transformation formula:
\begin{equation}
F^{(M, N)}_{\alpha, \beta} (z;w;a,b,c,d)
= e(u;\sqrt{q/t}\alpha)_{M-N}
F^{(N,M)}_{\beta, \alpha} (w;z;\sqrt{tq}/a, \sqrt{tq}/b, \sqrt{tq}/c, \sqrt{tq}/d).
\end{equation}
Relabeling $M$ and $N$ by $m$ and $n$, and 
replacing the variables and parameters by 
\begin{equation}
\begin{split}
z_i \rightarrow \sqrt{tq}x_i\ (i=1, \ldots, m), \quad w_k \rightarrow y_k\ (k=1, \ldots, n) ,\\
(a, b, c, d)\rightarrow (\sqrt{tq}/a_1,\sqrt{tq}/a_2,\sqrt{tq}/a_3,\sqrt{tq}/a_4),
\end{split}
\end{equation}
we obtain the following theorem.
\begin{theo}\label{Theo:trans form BC}
Let $a_0= \sqrt{a_1 a_2 a_3 a_4/q}$.
Take two sets of variables $x=(x_1, \ldots, x_m), y=(y_1, \ldots, y_n)$
and two multi-indices $\alpha=(\alpha_1, \ldots, \alpha_m) \in \mathbb{N}^m, \beta =(\beta_1, \ldots, \beta_n) \in \mathbb{N}^n$. 
Then the following identity holds:
\begin{align}
&\prod_{1 \le i \le m} \frac{\langle tq x_i/a_1, tqx_i/a_2, 
tqx_i/a_3, tqx_i/a_4 \rangle_{\alpha_i}}
{\langle tqx_i^2, t^2qx_i^2\rangle_{\alpha_i}}
\prod_{1 \le i < j \le m}
\frac{\langle tq^{\alpha_j+1}x_i x_j,  t^2q^{\alpha_j+1}x_i x_j\rangle_{\alpha_i}} 
{\langle tqx_i x_j,  t^2qx_i x_j\rangle_{\alpha_i}} \nonumber \\
&\cdot
\prod_{\substack{1 \le i \le m \\ 1 \le k \le n}}
\frac{\langle tx_i y_k,  tq^{1-\beta_k} x_i/y_k \rangle_{\alpha_i}}
{\langle t q^{\beta_k} x_i y_k,  tq x_i/y_k \rangle_{ \alpha_i}} 
\sum_{0 \le \mu^{-} \le \nu^{-} \le \nu^{+} \le \mu^{+} \le \alpha}
\Biggl( (-1)^{|\alpha|+|\nu^+| + |\nu^-|}
e(u; \sqrt{tq}/a_0)_{|\nu^+|-|\nu^-|} \nonumber \\
&\cdot \prod_{1 \le i \le m} 
\frac{\langle a_1 x_i, a_2 x_i, a_3 x_i, a_4 x_i\rangle_{\nu_i^{-}}}
{\langle tq x_i/a_1, tq x_i/a_2, tq x_i/a_3, tq x_i/a_4\rangle_{\nu_i^{+}}} 
\prod_{1 \le i \le j \le m}
\frac{\langle q^{\mu_i^- + \mu_j^-} x_i x_j, t^2q^{\mu_i^+ + \mu_j^+} x_i x_j\rangle}
{\langle x_i x_j, t^2 x_i x_j\rangle}
\nonumber \\
&\cdot \prod_{1 \le i < j \le m}
\frac{\langle q^{\mu_i^- - \mu_j^-} x_i/x_j, q^{\mu_i^+ - \mu_j^+} x_i/x_j,
q^{\nu_i^- - \nu_j^-} x_i/x_j, q^{\nu_i^+ - \nu_j^+} x_i/x_j\rangle
\langle tx_i x_j, qx_i x_j\rangle_{\nu_i^- + \nu_j^-}}
{\langle x_i/x_j, x_i/x_j, x_i/x_j, x_i/x_j\rangle 
\langle  t^2x_i x_j, tqx_i x_j \rangle_{\nu_i^+ + \nu_j^+}} \nonumber \\
&\cdot \prod_{1 \le i, j \le m}
\frac{\langle tq^{\mu_i^- + \mu_j^+} x_i x_j,
tq^{\nu_i^- + \nu_j^+} x_i x_j, q^{\mu_i^- - \mu_j^+} x_i/tx_j \rangle
\langle  tx_i x_j \rangle_{\mu_i^- + \nu_j^+}
\langle  t^2 x_i x_j \rangle_{\mu_i^+ + \nu_j^+}}
{\langle tx_i x_j,  tx_i x_j, x_i/tq^{\mu_j^+}x_j\rangle
\langle  qx_i x_j \rangle_{\mu_i^-+ \nu_j^-}
\langle  tqx_i x_j \rangle_{\mu_i^+ + \nu_j^-}} \nonumber \\
&\cdot \prod_{1 \le i, j \le m}
\frac{\langle  x_i x_j, x_i/q^{\nu_j^-}x_j , x_i/tq^{\alpha_j}x_j \rangle_{\mu_i^-}
\langle  tx_i x_j ,x_i/q^{\alpha_j}x_j \rangle_{\mu_i^+}
\langle  qx_i/tq^{\nu_j^+}x_j , x_i/q^{\nu_j^+}x_j \rangle_{\nu_i^-}
}
{\langle tq^{\alpha_j+1}x_i x_j , qx_i/tq^{\nu_j^+}x_j, qx_i/x_j \rangle_{\mu_i^-}
\langle t^2q^{\alpha_j+1} x_i x_j , qx_i/x_j \rangle_{\mu_i^+}
\langle qx_i/tq^{\mu_j^+}x_j, qx_i/x_j \rangle_{\nu_i^-}
}
\nonumber \\
&\cdot \prod_{1 \le i, j \le m}
\frac{\langle  x_i/q^{\mu_j^+}x_j \rangle_{\nu_i^+}}
{\langle qx_i/ x_j \rangle_{\nu_i^+}}
 \prod_{\substack{1\le i \le m\\ 1 \le k \le n}}
\frac{\langle q^{\beta_k} x_i y_k, q x_i/y_k \rangle_{\mu_i^-}
\langle t q^{\beta_k} x_i y_k,  tq x_i/y_k \rangle_{\mu_i^+}}
{\langle x_i y_k, q^{1-\beta_k} x_i/y_k \rangle_{\mu_i^-}
\langle tx_i y_k,  tq^{1-\beta_k} x_i/y_k \rangle_{\mu_i^+}} \Biggr)
\nonumber \\
&= e(u; \sqrt{tq}/a_0)_{|\alpha|-|\beta|} 
\prod_{1 \le k \le n} 
\frac{\langle a_1y_k, a_2y_k, a_3y_k, a_4y_k \rangle_{\beta_k}}
{\langle y_k^2, ty_k^2\rangle_{\beta_k}}
\prod_{1 \le k < l \le n}
\frac{\langle q^{\beta_l}y_k y_l,  tq^{\beta_l}y_k y_l\rangle_{\beta_k}}
{\langle y_k y_l,  ty_k y_l\rangle_{\beta_k}} \nonumber \\
&\cdot \prod_{\substack{1 \le k \le n \\ 1 \le i \le m}}
\frac{\langle ty_k x_i,  q^{-\alpha_i}y_k/x_i \rangle_{\beta_k}}
{\langle tq^{\alpha_i}y_kx_i, y_k/x_i \rangle_{\beta_k}}
\sum_{0 \le \lambda^{-} \le \kappa^{-} \le \kappa^{+} \le \lambda^{+} \le \beta}
\Biggl( (-1)^{|\beta|+|\kappa^+| + |\kappa^-|}
e(u; \sqrt{q/t} a_0)_{|\kappa^+|-|\kappa^-|}  \nonumber \\
&\cdot
\prod_{1 \le k \le n} \frac{\langle y_k/a_1, y_k/a_2,
y_k/a_3, y_k/a_4 \rangle_{\kappa_k^{-}}} 
{\langle a_1y_k, a_2y_k, a_3y_k, a_4y_k
\rangle_{\kappa_k^{+}}} 
\prod_{1 \le k \le l \le n}
\frac{\langle q^{\lambda_k^- + \lambda_l^-} y_k y_l/tq,
tq^{\lambda_k^+ + \lambda_l^+} y_k y_l/q\rangle}
{\langle y_k y_l/tq, t y_k y_l/q\rangle} \nonumber \\
&\cdot \prod_{1 \le k < l \le n}
\frac{\langle q^{\lambda_k^- - \lambda_l^-} y_k/y_l,
q^{\lambda_k^+ - \lambda_l^+} y_k/y_l,
 q^{\kappa_k^- - \kappa_l^-} y_k/y_l,
q^{\kappa_k^+ - \kappa_l^+} y_k/y_l\rangle
\langle  y_k y_l/q, y_k y_l/t \rangle_{\kappa_k^- + \kappa_l^-}}
{\langle y_k/y_l,y_k/y_l,y_k/y_l,y_k/y_l\rangle
\langle ty_k y_l/q, y_k y_l \rangle_{\kappa_k^+ + \kappa_l^+}} \nonumber \\
&\cdot \prod_{1 \le k, l \le n}
\frac{\langle q^{\lambda_k^- + \lambda_l^+} y_k y_l/q,
q^{\kappa_k^- + \kappa_l^+} y_k y_l/q, 
q^{\lambda_k^- - \lambda_l^+} y_k/ty_l\rangle
\langle  y_k y_l/q \rangle_{\lambda_k^- + \kappa_l^+}
\langle  ty_k y_l/q \rangle_{\lambda_k^+ + \kappa_l^+}}
{\langle y_k y_l/q,  y_k y_l/q,  y_k/tq^{\lambda_l^+}y_l \rangle
\langle  y_k y_l/t \rangle_{\lambda_k^-+ \kappa_l^-}
\langle  y_k y_l \rangle_{\lambda_k^+ + \kappa_l^-}} \nonumber \\
&\cdot \prod_{1 \le k, l \le n}
\frac{\langle  y_k y_l/tq, y_k/q^{\kappa_l^-}y_l , y_k/tq^{\beta_l}y_l \rangle_{\lambda_k^-}
\langle  y_k y_l/q,  y_k/q^{\beta_l}y_l \rangle_{\lambda_k^+}
\langle  qy_k/tq^{\kappa_l^+}y_l , y_k/q^{\kappa_l^+}y_l \rangle_{\kappa_k^-}
}
{\langle  q^{\beta_l}y_k y_l, qy_k/tq^{\kappa_l^+}y_l, qy_k/ y_l \rangle_{\lambda_k^-}
\langle  tq^{\beta_l} y_k y_l, qy_k/ y_l  \rangle_{\lambda_k^+}
\langle qy_k/tq^{\lambda_l^+}y_l, qy_k/ y_l \rangle_{\kappa_k^-}
}  \nonumber \\
&\cdot
\prod_{1 \le k, l \le n}
\frac{\langle  y_k/q^{\lambda_l^+}y_l \rangle_{\kappa_k^+}}
{\langle qy_k/ y_l \rangle_{\kappa_k^+}}
\prod_{\substack{ 1 \le k \le n \\ 1\le i \le m}}
\frac{\langle q^{\alpha_i} y_k x_i ,  y_k/tx_i \rangle_{\lambda_k^-}
\langle tq^{\alpha_i}y_kx_i, y_k/x_i \rangle_{\lambda_k^+}}
{\langle y_kx_i , y_k/tq^{\alpha_i}x_i \rangle_{\lambda_k^-}
\langle ty_k x_i,  q^{-\alpha_i}y_k/x_i \rangle_{\lambda_k^+}} \Biggr). \label{eq:principal}
\end{align}
\end{theo}
We give some remarks of Theorem \ref{Theo:trans form BC}.
As the special case $n=0$,
we obtain the following summation formula:
\begin{align}
&\sum_{0 \le \mu^{-} \le \nu^{-} \le \nu^{+} \le \mu^{+} \le \alpha}
\Biggl( (-1)^{|\alpha|+|\nu^+| + |\nu^-|}
e(u; \sqrt{tq}/a_0)_{|\nu^+|-|\nu^-|} \nonumber \\
&\cdot \prod_{1 \le i \le m} 
\frac{\langle a_1 x_i, a_2 x_i, a_3 x_i, a_4 x_i\rangle_{\nu_i^{-}}}
{\langle tq x_i/a_1, tq x_i/a_2, tq x_i/a_3, tq x_i/a_4\rangle_{\nu_i^{+}}} 
\prod_{1 \le i \le j \le m}
\frac{\langle q^{\mu_i^- + \mu_j^-} x_i x_j, t^2q^{\mu_i^+ + \mu_j^+} x_i x_j\rangle}
{\langle x_i x_j, t^2 x_i x_j\rangle}
\nonumber \\
&\cdot \prod_{1 \le i < j \le m}
\frac{\langle q^{\mu_i^- - \mu_j^-} x_i/x_j, q^{\mu_i^+ - \mu_j^+} x_i/x_j,
q^{\nu_i^- - \nu_j^-} x_i/x_j, q^{\nu_i^+ - \nu_j^+} x_i/x_j\rangle
\langle tx_i x_j, qx_i x_j\rangle_{\nu_i^- + \nu_j^-}}
{\langle x_i/x_j, x_i/x_j, x_i/x_j, x_i/x_j\rangle 
\langle  t^2x_i x_j, tqx_i x_j \rangle_{\nu_i^+ + \nu_j^+}} \nonumber \\
&\cdot \prod_{1 \le i, j \le m}
\frac{\langle tq^{\mu_i^- + \mu_j^+} x_i x_j,
tq^{\nu_i^- + \nu_j^+} x_i x_j, q^{\mu_i^- - \mu_j^+} x_i/tx_j \rangle
\langle  tx_i x_j \rangle_{\mu_i^- + \nu_j^+}
\langle  t^2 x_i x_j \rangle_{\mu_i^+ + \nu_j^+}
\langle  x_i/q^{\mu_j^+}x_j \rangle_{\nu_i^+}}
{\langle tx_i x_j,  tx_i x_j, x_i/tq^{\mu_j^+}x_j\rangle
\langle  qx_i x_j \rangle_{\mu_i^-+ \nu_j^-}
\langle  tqx_i x_j \rangle_{\mu_i^+ + \nu_j^-}
\langle qx_i/ x_j \rangle_{\nu_i^+}} \nonumber \\
&\cdot \prod_{1 \le i, j \le m}
\frac{\langle  x_i x_j, x_i/q^{\nu_j^-}x_j , x_i/tq^{\alpha_j}x_j \rangle_{\mu_i^-}
\langle  tx_i x_j ,x_i/q^{\alpha_j}x_j \rangle_{\mu_i^+}
\langle  qx_i/tq^{\nu_j^+}x_j , x_i/q^{\nu_j^+}x_j \rangle_{\nu_i^-}
}
{\langle tq^{\alpha_j+1}x_i x_j , qx_i/tq^{\nu_j^+}x_j, qx_i/x_j \rangle_{\mu_i^-}
\langle t^2q^{\alpha_j+1} x_i x_j , qx_i/x_j \rangle_{\mu_i^+}
\langle qx_i/tq^{\mu_j^+}x_j, qx_i/x_j \rangle_{\nu_i^-}
} \Biggr)
\nonumber \\
&= e(u; \sqrt{tq}/a_0)_{|\alpha|} 
\prod_{1 \le i \le m} \frac{\langle tqx_i^2, t^2qx_i^2\rangle_{\alpha_i}}
{\langle tq x_i/a_1, tqx_i/a_2, tqx_i/a_3, tqx_i/a_4 \rangle_{\alpha_i}}
 \prod_{1 \le i < j \le m}\frac
{\langle tqx_i x_j,  t^2qx_i x_j\rangle_{\alpha_i}} 
{\langle tq^{\alpha_j+1}x_i x_j,  t^2q^{\alpha_j+1}x_ix_j\rangle_{\alpha_i}}. 
\end{align}

Note that the $q$-$\text{Saalsch}\ddot{\text{u}}\text{tz}$ sum (\ref{eq:Saal}) implies
\begin{align}
&e(u; \sqrt{tq}/a_0)_{|\alpha|-|\beta|}
e(u; \sqrt{q/t}a_0)_{|\kappa^+|-|\kappa^-|}\nonumber \\
&= \sum_{r=0}^{|\kappa^+|-|\kappa^-|} \Biggl(
\frac{\langle q^{|\kappa^-|-|\kappa^+|} \rangle_{r}}{\langle q \rangle_{r}}
e(t^{-\frac{1}{2}}q^{\frac{1}{2}(|\kappa^+|-|\kappa^-|)}a_0;
t^{-\frac{1}{2}}
q^{\frac{1}{2}(|\kappa^-|-|\kappa^+|-2|\alpha|+2|\beta|)}a_0)_{r} \nonumber \\
& \cdot e(u;\sqrt{q/t}/a_0)_{|\alpha|-|\beta|+|\kappa^+|-|\kappa^-|-r} \Biggr).
\end{align}
Hence,
the right-hand side of (\ref{eq:principal}) is also expressed as
\begin{align}
&\prod_{1 \le k \le n} 
\frac{\langle a_1y_k, a_2y_k, a_3y_k, a_4y_k \rangle_{\beta_k}}
{\langle y_k^2, ty_k^2\rangle_{\beta_k}}
\prod_{1 \le k < l \le n}
\frac{\langle q^{\beta_l}y_k y_l,  tq^{\beta_l}y_k y_l\rangle_{\beta_k}}
{\langle y_k y_l,  ty_k y_l\rangle_{\beta_k}}
\prod_{\substack{1 \le k \le n \\ 1 \le i \le m}}
\frac{\langle ty_k x_i,  q^{-\alpha_i}y_k/x_i \rangle_{\beta_k}}
{\langle tq^{\alpha_i}y_kx_i, y_k/x_i \rangle_{\beta_k}} \nonumber \\
&\cdot \sum_{0 \le \lambda^{-} \le \kappa^{-} \le \kappa^{+} \le \lambda^{+} \le \beta}
\Biggl(
\sum_{r=0}^{|\kappa^+|-|\kappa^-|}\Biggl(
(-1)^{|\beta|+|\kappa^+| + |\kappa^-|} 
\frac{\langle q^{|\kappa^-|-|\kappa^+|} \rangle_{r}}{\langle q \rangle_{r}} \nonumber \\
&\cdot e(t^{-\frac{1}{2}}q^{\frac{1}{2}(|\kappa^+|-|\kappa^-|)}a_0;
t^{-\frac{1}{2}}q^{\frac{1}{2}(|\kappa^-|-|\kappa^+|-2|\alpha|+2|\beta|)} a_0)_{r}
e(u;\sqrt{tq}/a_0)_{|\alpha|-|\beta|+|\kappa^+|-|\kappa^-|-r} \nonumber \\
&\cdot \prod_{1 \le k \le n} \frac{\langle y_k/a_1, y_k/a_2,
y_k/a_3, y_k/a_4 \rangle_{\kappa_k^{-}}} 
{\langle a_1y_k, a_2y_k, a_3y_k, a_4y_k
\rangle_{\kappa_k^{+}}} 
\prod_{1 \le k \le l \le n}
\frac{\langle q^{\lambda_k^- + \lambda_l^- } y_k y_l/tq,
tq^{\lambda_k^+ + \lambda_l^+ } y_k y_l/q\rangle}
{\langle y_k y_l/tq, t y_k y_l/q\rangle} \nonumber \\
&\cdot \prod_{1 \le k < l \le n}
\frac{\langle q^{\lambda_k^- - \lambda_l^-} y_k/y_l,
q^{\lambda_k^+ - \lambda_l^+} y_k/y_l,
 q^{\kappa_k^- - \kappa_l^-} y_k/y_l,
q^{\kappa_k^+ - \kappa_l^+} y_k/y_l\rangle
\langle  y_k y_l/q, y_k y_l/t \rangle_{\kappa_k^- + \kappa_l^-}}
{\langle y_k/y_l,y_k/y_l,y_k/y_l,y_k/y_l\rangle
\langle ty_k y_l/q, y_k y_l \rangle_{\kappa_k^+ + \kappa_l^+}} \nonumber \\
&\cdot \prod_{1 \le k, l \le n}
\frac{\langle q^{\lambda_k^- + \lambda_l^+} y_k y_l/q,
q^{\kappa_k^- + \kappa_l^+} y_k y_l/q, 
q^{\lambda_k^- - \lambda_l^+} y_k/ty_l\rangle
\langle  y_k y_l/q \rangle_{\lambda_k^- + \kappa_l^+}
\langle  ty_k y_l/q \rangle_{\lambda_k^+ + \kappa_l^+}}
{\langle y_k y_l/q,  y_k y_l/q,  y_k/tq^{\lambda_l^+}y_l \rangle
\langle  y_k y_l/t \rangle_{\lambda_k^-+ \kappa_l^-}
\langle  y_k y_l \rangle_{\lambda_k^+ + \kappa_l^-}} \nonumber \\
&\cdot \prod_{1 \le k, l \le n}
\frac{\langle  y_k y_l/tq, y_k/q^{\kappa_l^-}y_l , y_k/tq^{\beta_l}y_l \rangle_{\lambda_k^-}
\langle  y_k y_l/q,  y_k/q^{\beta_l}y_l \rangle_{\lambda_k^+}
\langle  qy_k/tq^{\kappa_l^+}y_l , y_k/q^{\kappa_l^+}y_l \rangle_{\kappa_k^-}
\langle  y_k/q^{\lambda_l^+}y_l \rangle_{\kappa_k^+}}
{\langle  q^{\beta_l}y_k y_l, qy_k/tq^{\kappa_l^+}y_l, qy_k/ y_l \rangle_{\lambda_k^-}
\langle  tq^{\beta_l} y_k y_l, qy_k/ y_l  \rangle_{\lambda_k^+}
\langle qy_k/tq^{\lambda_l^+}y_l, qy_k/ y_l \rangle_{\kappa_k^-}
\langle qy_k/ y_l \rangle_{\kappa_k^+}}  \nonumber \\
&\cdot
\prod_{\substack{ 1 \le k \le n \\ 1\le i \le m}}
\frac{\langle q^{\alpha_i} y_k x_i ,  y_k/tx_i \rangle_{\lambda_k^-}
\langle tq^{\alpha_i}y_kx_i, y_k/x_i \rangle_{\lambda_k^+}}
{\langle y_kx_i , y_k/tq^{\alpha_i}x_i \rangle_{\lambda_k^-}
\langle ty_k x_i,  q^{-\alpha_i}y_k/x_i \rangle_{\lambda_k^+}}\Biggr) \Biggr). \label{eq:r.h.s of principal}
\end{align}

\subsection{Type $C$ case}
\label{subsec:type C}
In the previous subsection, we derived the transformation formula of type $BC$
by a method of principal specializations.
In this subsection, we apply the same method to derive another type of transformation formula.
We specialize in advance the parameters of (\ref{eq:main theo q})
so that $(c, d, t) = (q^{1/2}, -q^{1/2}, q)$:
\begin{align}
&\sum_{\substack{I \subset \{ 1, \ldots , m\} \\ \epsilon_i = \pm1 (i \in I)}} 
\Biggl( (-1)^{|I|}
\prod_{i \in I} \frac{\sqrt{-1}
\langle ax_i^{\epsilon_i}, bx_i^{\epsilon_i} \rangle}
{\langle x_i^{2\epsilon_i}\rangle} 
\prod_{\substack{i, j \in I\\ i < j}}
\frac{\langle q^2 x_i^{\epsilon_i} x_{j}^{\epsilon_{j}}\rangle}
{\langle x_i^{\epsilon_i}x_{j}^{\epsilon_{j}}\rangle}
\prod_{\substack{i \in I\\ j \in I^c}}
\frac{\langle q x_i^{\epsilon_i} x_{j}^{\pm 1}\rangle}{\langle x_i^{\epsilon_i}x_{j}^{\pm 1} \rangle}
\prod_{\substack{i \in I \\ 1 \le k \le n}} \frac{\langle x_i^{\epsilon_i}y_k^{\pm 1} \rangle}
{\langle qx_i^{\epsilon_i}y_k^{\pm 1}\rangle} \nonumber \\
& \cdot \sum_{\substack{ J \subset I^c \\  \delta_i = \pm1 (i \in J)}} 
e(u; \alpha)_{|I^c|-|J|} 
\prod_{i \in J} 
\frac{\sqrt{-1} \langle a x_i^{\delta_i}, b x_i^{\delta_i} \rangle}
{\langle x_i^{2\delta_i}\rangle} 
\prod_{\substack{i \in J\\ j \in I^c\backslash J}} 
\frac{\langle q x_i^{\delta_i} x_{j}^{\pm 1}\rangle}
{\langle x_i^{\delta_i}x_{j}^{\pm 1} \rangle} \Biggr)  \nonumber \\
&=e(u; \alpha)_{m-n}
\sum_{\substack{K \subset \{ 1, \ldots , n\} \\ \epsilon_k=\pm 1 (k \in K)}} 
\Biggl( (-1)^{|K|}\prod_{k \in K}
\frac{-\sqrt{-1}
\langle qy_k^{\epsilon_k}/a, qy_k^{\epsilon_k}/b \rangle}
{\langle y_k^{2\epsilon_k}\rangle} 
\prod_{\substack{k, l \in K\\ k < l}} 
\frac{\langle q^2y_k^{\epsilon_k} y_{l}^{\epsilon_{l}}\rangle}
{\langle y_k^{\epsilon_k}y_{l}^{\epsilon_{l}}\rangle}
\prod_{\substack{k \in K\\ l \in K^c}}
\frac{\langle q y_k^{\epsilon_k} y_{l}^{\pm 1}\rangle}{\langle y_k^{\epsilon_k}y_{l}^{\pm 1} \rangle}\nonumber \\ 
&\cdot
\prod_{\substack{k \in K \\ 1 \le i \le m}} 
\frac{\langle y_k^{\epsilon_k}x_i^{\pm1} \rangle}
{\langle qy_k^{\epsilon_k} x_i^{\pm1}\rangle}
\sum_{\substack{ L \subset K^c \\  \delta_k=\pm 1 (k \in L)}} 
e(u; q/\alpha)_{|K^c|-|L|}
\prod_{k \in L} 
\frac{-\sqrt{-1}\langle qy_k^{\delta_k}/a, qy_k^{\delta_k}/b\rangle}
{\langle y_k^{2\delta_k}\rangle}
\prod_{\substack{k \in L\\ l \in K^c\backslash L}} 
\frac{\langle q y_k^{\delta_k} y_{l}^{\pm 1}\rangle}
{\langle y_k^{\delta_k}y_{l}^{\pm 1} \rangle}\Biggr) 
\label{eq:t=q on main theo}.
\end{align}
Note that $\alpha = \sqrt{-ab}$.
In this setting, the internal sum of each side simplifies drastically.
\begin{lemm}
\begin{align}
&\sum_{\substack{ J \subset I^c \\  \delta_i = \pm1 (i \in J)}} 
e(u; \alpha)_{|I^c|-|J|} 
\prod_{i \in J} 
\frac{\sqrt{-1} \langle a x_i^{\delta_i}, b x_i^{\delta_i} \rangle}
{\langle x_i^{2\delta_i}\rangle} 
\prod_{\substack{i \in J\\ j \in I^c\backslash J}} 
\frac{\langle q x_i^{\delta_i} x_{j}^{\pm 1}\rangle}
{\langle x_i^{\delta_i}x_{j}^{\pm 1} \rangle} 
=\left(u+\frac{1}{u} \right)^{|I^c|}, \\
&\sum_{\substack{ L \subset K^c \\  \delta_k = \pm1 (k \in L)}} 
e(u; q/\alpha)_{|K^c|-|L|}
\prod_{k \in L} 
\frac{-\sqrt{-1}\langle qy_k^{\delta_k}/a, qy_k^{\delta_k}/b\rangle}
{\langle y_k^{2\delta_k}\rangle}
\prod_{\substack{k \in L\\ l \in K^c\backslash L}} 
\frac{\langle q y_k^{\delta_k} y_{l}^{\pm 1}\rangle}
{\langle y_k^{\delta_k}y_{l}^{\pm 1} \rangle}
 = \left(u+\frac{1}{u} \right)^{|K^c|}.
\end{align}
\end{lemm}
We can prove this lemma in the same way as in
Theorem \ref{Theo:Main} by analyzing the residues.
From this lemma, if $u= \sqrt{-1}$,
(\ref{eq:t=q on main theo}) reduces to
\begin{align}
&\sum_{\substack{\epsilon_i = \pm1 \\ 1 \le i \le m}} 
(-1)^{m}
\prod_{1 \le i \le m} \frac{\sqrt{-1} \langle ax_i^{\epsilon_i}, bx_i^{\epsilon_i} \rangle}
{\langle x_i^{2\epsilon_i}\rangle} 
\prod_{1 \le i < j \le m}
\frac{\langle q^2 x_i^{\epsilon_i} x_{j}^{\epsilon_{j}}\rangle}
{\langle x_i^{\epsilon_i}x_{j}^{\epsilon_{j}}\rangle}
\prod_{\substack{1 \le i \le m \\ 1 \le k \le n}}
 \frac{\langle x_i^{\epsilon_i}y_k^{\pm 1} \rangle}
{\langle qx_i^{\epsilon_i}y_k^{\pm 1}\rangle} \nonumber \\
&=e(\sqrt{-1}; \alpha)_{m-n}
\sum_{\substack{ \epsilon_k=\pm 1 \\ 1 \le k \le n}} 
(-1)^{n}
\prod_{1 \le k \le n}\frac{-\sqrt{-1}
\langle qy_k^{\epsilon_k}/a, qy_k^{\epsilon_k}/b \rangle}
{\langle y_k^{2\epsilon_k}\rangle} 
\prod_{\substack{1 \le k < l \le n}} 
\frac{\langle q^2y_k^{\epsilon_k} y_{l}^{\epsilon_{l}}\rangle}
{\langle y_k^{\epsilon_k}y_{l}^{\epsilon_{l}}\rangle}
\prod_{\substack{1 \le k \le n \\ 1 \le i \le m}} 
\frac{\langle y_k^{\epsilon_k}x_i^{\pm1} \rangle}
{\langle qy_k^{\epsilon_k} x_i^{\pm1}\rangle}. \label{eq:q^2 ver}
\end{align}
Since $e(\sqrt{-1};\alpha)_{m-n} = (\sqrt{-1})^{m-n}(-1)^m
\frac{\langle ab \rangle_{q^2, m}}{\langle q^{2-2m}/ab \rangle_{q^2, n}}$,
(\ref{eq:q^2 ver}) is equal to
\begin{align}
&\sum_{\substack{\epsilon_i = \pm1 \\ 1 \le i \le m}} 
\prod_{1 \le i \le m} \frac{\langle ax_i^{\epsilon_i}, bx_i^{\epsilon_i} \rangle}
{\langle x_i^{2\epsilon_i}\rangle} 
\prod_{1 \le i < j \le m}
\frac{\langle q^2 x_i^{\epsilon_i} x_{j}^{\epsilon_{j}}\rangle}
{\langle x_i^{\epsilon_i}x_{j}^{\epsilon_{j}}\rangle}
\prod_{\substack{1 \le i \le m \\ 1 \le k \le n}}
 \frac{\langle x_i^{\epsilon_i}y_k^{\pm 1} \rangle}
{\langle qx_i^{\epsilon_i}y_k^{\pm 1}\rangle} \nonumber \\
&=\frac{\langle ab \rangle_{q^2, m}}{\langle q^{2-2m}/ab \rangle_{q^2, n}}
\sum_{\substack{ \epsilon_k=\pm 1 \\ 1 \le k \le n }} 
\prod_{1 \le k \le n}\frac{
\langle qy_k^{\epsilon_k}/a, qy_k^{\epsilon_k}/b \rangle}
{\langle y_k^{2\epsilon_k}\rangle} 
\prod_{\substack{1 \le k < l \le n}} 
\frac{\langle q^2y_k^{\epsilon_k} y_{l}^{\epsilon_{l}}\rangle}
{\langle y_k^{\epsilon_k}y_{l}^{\epsilon_{l}}\rangle}
\prod_{\substack{1 \le k \le n \\ 1 \le i \le m}} 
\frac{\langle y_k^{\epsilon_k}x_i^{\pm1} \rangle}
{\langle qy_k^{\epsilon_k} x_i^{\pm1}\rangle}. \label{eq:set ver C}
\end{align}
Let $I^-, I^+, K^-, K^+$ be 
\begin{equation}
\begin{split}
I^-= \{i | 1 \le i \le m, \epsilon_i = -1 \},\quad 
I^+= \{i | 1 \le i \le m, \epsilon_i = 1 \}, \\
K^-= \{k | 1 \le k \le n, \epsilon_k = -1 \},\quad 
K^+= \{k | 1 \le k \le n, \epsilon_k = 1 \}.
\end{split}
\end{equation}
Using this notation and
replacing the parameter $q$ with $q^{1/2}$,
we obtain
\begin{align}
&\sum_{\substack{I^- \sqcup I^+ \\=\{1, \ldots, m \}}} 
\Biggl(
\prod_{i \in I^+}
\frac{\langle ax_i, bx_i\rangle}
{\langle x_i^{2}\rangle} 
\prod_{i \in I^-}
\frac{-\langle x_i/a, x_i/b\rangle}
{\langle x_i^{2}\rangle}
\prod_{\substack{i, j \in I^+ \\ i<j}}
\frac{\langle q x_i x_j \rangle}
{\langle x_i x_j \rangle}
\prod_{\substack{i, j \in I^- \\ i<j}}
\frac{\langle x_i x_j /q\rangle}
{\langle x_i x_j \rangle} \nonumber \\
&\cdot \prod_{\substack{i \in I^+ \\ j \in I^-}}
\frac{\langle q x_i/x_j \rangle}
{\langle x_i/x_j \rangle}
\prod_{\substack{i \in I^- \\ 1 \le k \le n}}
 \frac{\langle \sqrt{q} x_i y_k^{\pm 1} \rangle}
{\langle x_i y_k^{\pm 1}/\sqrt{q}\rangle} \Biggr)
\nonumber \\
&=\frac{\langle ab \rangle_m}{\langle q^{1-m}/ab\rangle_n}
\prod_{\substack{1 \le i \le m \\ 1 \le k \le n}}
\frac{\langle y_k/\sqrt{q}x_i \rangle}{\langle \sqrt{q}y_k/x_i \rangle}
\sum_{\substack{K^- \sqcup K^+ \\ =\{1, \ldots, n \}}} 
\Biggl( \prod_{k \in K^+}\frac{
\langle \sqrt{q}y_k/a, \sqrt{q}y_k/b \rangle}
{\langle y_k^{2}\rangle}
\prod_{k \in K^-}
\frac{-\langle ay_k/\sqrt{q}, by_k/\sqrt{q} \rangle}
{\langle y_k^{2}\rangle} \nonumber \\
&\cdot \prod_{\substack{k, l \in K^+ \\ k < l}} 
\frac{\langle qy_k y_l\rangle}
{\langle y_k y_l\rangle}
\prod_{\substack{k, l \in K^- \\ k < l}} 
\frac{\langle y_k y_l/q \rangle}
{\langle y_k y_l \rangle}
\prod_{\substack{k \in K^+ \\ l \in K^-}} 
\frac{\langle qy_k/y_l\rangle}
{\langle y_k/y_l\rangle}
\prod_{\substack{k \in K^- \\ 1 \le i \le m}} 
\frac{\langle \sqrt{q} y_k x_i^{\pm1} \rangle}
{\langle y_k x_i^{\pm1}/\sqrt{q} \rangle}
\Biggr).
\end{align}
Specializing this formula as $x=p_{\alpha}(z;q), 
y=p_{\beta}(w; q)\ (\alpha \in \mathbb{N}^M,
\beta \in \mathbb{N}^N, |\alpha|=m, |\beta|=n)$,
we obtain
\begin{align}
&\sum_{0 \le \mu \le  \alpha}
\Biggl( \prod_{1 \le i \le M}
\frac{\langle z_i/a, z_i/b \rangle_{\mu_i}}
{\langle az_i, bz_i \rangle_{\mu_i}}
\prod_{1\le i<j \le M}
\frac{\langle q^{\mu_i- \mu_j}z_i/z_j\rangle}{\langle z_i/z_j \rangle}
\prod_{1 \le i \le j \le M} 
\frac{\langle q^{\mu_i+\mu_j}z_iz_j/q \rangle}
{\langle z_i z_j/q \rangle} \nonumber \\
& \cdot \prod_{1\le i, j \le M}
\frac{\langle  z_iz_j/q, z_i/q^{\alpha_j} z_j\rangle_{\mu_i}}
{\langle q^{\alpha_j} z_i z_j, qz_i/z_j \rangle_{\mu_i}}
\prod_{\substack{1\le i \le M \\ 1 \le k \le N}}
\frac{\langle q^{\beta_k} z_i w_k/\sqrt{q}, \sqrt{q}z_i/w_k\rangle_{\mu_i}}
{\langle z_i w_k/\sqrt{q}, \sqrt{q}z_i/q^{\beta_k} w_k \rangle_{\mu_i}}
\Biggr) \nonumber \\
&= \frac{\langle ab \rangle_{|\alpha|}}
{\langle q^{1- |\alpha|}/ab \rangle_{|\beta|}}
\frac{\prod_{1 \le k \le N}\langle \sqrt{q}w_k/a, \sqrt{q}w_k/b \rangle_{\beta_k}}
{\prod_{1 \le i \le M} \langle az_i, bz_i \rangle_{\alpha_i}}
\frac{\prod_{1 \le i, j \le M} \langle z_i z_j\rangle_{\alpha_i}}
{\prod_{1 \le i < j \le M} \langle z_i z_j\rangle_{\alpha_i+\alpha_j}}
\frac{\prod_{1 \le k < l \le N} \langle w_k w_l \rangle_{\beta_k+\beta_l}}
{\prod_{1 \le k, l \le N} \langle w_k w_l \rangle_{\beta_k}} \nonumber \\
&\cdot \prod_{\substack{1\le k \le N \\ 1 \le i \le M}}
\frac{\langle \sqrt{q}w_k/q^{\alpha_i} z_i\rangle_{\beta_k}}
{\langle \sqrt{q}w_k/z_i \rangle_{\beta_k}}
\sum_{0 \le \nu \le \beta}
\Biggl( \prod_{1 \le k \le N}
\frac{\langle aw_k/\sqrt{q}, bw_k/\sqrt{q} \rangle_{\nu_k}}
{\langle \sqrt{q}w_k/a, \sqrt{q}w_k/b \rangle_{\nu_k}}
\prod_{1\le k<l \le N}
\frac{\langle q^{\nu_k- \nu_l}w_k/w_l\rangle}{\langle w_k/w_l \rangle}
 \nonumber \\
&\cdot \prod_{1 \le k \le l \le N} 
\frac{\langle q^{\nu_k+\nu_l}w_k w_l/q \rangle}
{\langle w_k w_l/q \rangle}
\prod_{1\le k, l \le N}
\frac{\langle  w_kw_l/q, w_k/q^{\beta_l} w_l\rangle_{\nu_k}}
{\langle q^{\beta_l} w_k w_l, qw_k/w_l \rangle_{\nu_k}}
\prod_{\substack{1\le k \le N \\ 1 \le i \le M}}
\frac{\langle q^{\alpha_i} w_k z_i/\sqrt{q}, \sqrt{q}w_k/z_i\rangle_{\nu_k}}
{\langle w_kz_i/\sqrt{q}, \sqrt{q}w_k/q^{\alpha_i} z_i \rangle_{\nu_k}}
\Biggr).
\end{align}
Relabeling $M$ and $N$ by $m$ and $n$, and 
replacing the variables and parameters by 
\begin{align}
z_i \rightarrow \sqrt{q}x_i\ (i=1, \ldots, m), \quad w_k \rightarrow y_k\ (k=1, \ldots, n), \quad 
(a, b)\rightarrow (\sqrt{q}/a_1,\sqrt{q}/a_2),
\end{align}
we obtain a transformation formula of type $C$, due to Rosengren \cite[Corollary 4.4]{R1}
.
\begin{theo}\label{Theo:trans form C}
For $\alpha \in \mathbb{N}^m$ and $\beta \in \mathbb{N}^n$,
the following identity holds:
\begin{align}
&\sum_{0 \le \mu \le  \alpha}
\Biggl( \prod_{1 \le i \le m}
\frac{\langle a_1x_i, a_2x_i \rangle_{\mu_i}}
{\langle qx_i/a_1, qx_i/a_2 \rangle_{\mu_i}}
\prod_{1\le i<j \le m}
\frac{\langle q^{\mu_i- \mu_j}x_i/x_j\rangle}{\langle x_i/x_j \rangle}
\prod_{1 \le i \le j \le m} 
\frac{\langle q^{\mu_i+\mu_j}x_ix_j \rangle}
{\langle x_i x_j \rangle} \nonumber \\
& \cdot \prod_{1\le i, j \le m}
\frac{\langle  x_ix_j, x_i/q^{\alpha_j} x_j\rangle_{\mu_i}}
{\langle q^{\alpha_j+1} x_i x_j, qx_i/x_j \rangle_{\mu_i}}
\prod_{\substack{1\le i \le m \\ 1 \le k \le n}}
\frac{\langle q^{\beta_k} x_i y_k, q x_i/y_k\rangle_{\mu_i}}
{\langle x_i y_k, q^{1-\beta_k}x_i/y_k \rangle_{\mu_i}}
\Biggr) \nonumber \\
&= \frac{\langle q/a_1a_2 \rangle_{|\alpha|}}
{\langle a_1 a_2/q^{|\alpha|} \rangle_{|\beta|}}
\frac{\prod_{1 \le k \le n}\langle a_1 y_k, a_2y_k \rangle_{\beta_k}}
{\prod_{1 \le i \le m} \langle qx_i/a_1, qx_i/a_2 \rangle_{\alpha_i}}
\frac{\prod_{1 \le i, j \le m} \langle qx_i x_j\rangle_{\alpha_i}}
{\prod_{1 \le i < j \le m} \langle qx_i x_j\rangle_{\alpha_i+\alpha_j}}
\frac{\prod_{1 \le k < l \le n} \langle y_k y_l \rangle_{\beta_k+\beta_l}}
{\prod_{1 \le k, l \le n} \langle y_k y_l \rangle_{\beta_k}} \nonumber \\
&\cdot \prod_{\substack{1\le k \le n \\ 1 \le i \le m}}
\frac{\langle y_k/q^{\alpha_i} x_i\rangle_{\beta_k}}
{\langle y_k/x_i \rangle_{\beta_k}}
\sum_{0 \le \nu \le \beta}
\Biggl( \prod_{1 \le k \le n}
\frac{\langle y_k/a_1, y_k/a_2 \rangle_{\nu_k}}
{\langle a_1y_k, a_2 y_k\rangle_{\nu_k}}
\prod_{1\le k<l \le n}
\frac{\langle q^{\nu_k- \nu_l}y_k/y_l\rangle}{\langle y_k/y_l \rangle}
 \nonumber \\
&\cdot \prod_{1 \le k \le l \le n} 
\frac{\langle q^{\nu_k+\nu_l}y_k y_l/q \rangle}
{\langle y_k y_l/q \rangle}
\prod_{1\le k, l \le n}
\frac{\langle  y_k y_l/q, y_k/q^{\beta_l} y_l\rangle_{\nu_k}}
{\langle q^{\beta_l} y_k y_l, qy_k/y_l \rangle_{\nu_k}}
\prod_{\substack{1\le k \le n \\ 1 \le i \le m}}
\frac{\langle q^{\alpha_i} y_k x_i, y_k/x_i\rangle_{\nu_k}}
{\langle y_k x_i, y_k/q^{\alpha_i} x_i \rangle_{\nu_k}}
\Biggr). 
\end{align}
\end{theo}
Rosengren derived this result from 
Gustafson's summation formula of multilateral basic hypergeometric series for type $C$.
We remark that Lassalle has derived a special case of Theorem \ref{Theo:trans form C} from a rational function identity by the method of principal specialization
\cite[Theorem 11]{L1}.
His rational function identity (Theorem 6) corresponds to (\ref{eq:set ver C})
with $a=q, b=-q$.

\section{New family of $q$-difference operators}
\label{sec:New family}
In the $A$ type case,
it is known that there exists an explicit operator 
$\mathcal{H}_A^{x}(u;q, t)$ satisfying the following equation \cite{N2}:
\begin{equation}
(u;q)_{\infty} \mathcal{H}_A^{x}(u;q, t) \Psi_A(x;y)
=(t^m q^n u;q)_{\infty} \mathcal{D}_A^{y}(u ; t,q) \Psi_A(x; y), 
\label{eq:dual Cauchy type A}
\end{equation}
where $\Psi_A(x; y)=\prod_{1 \le i \le m} \prod_{1 \le k \le n}(x_i -y_k)$
is the kernel function of dual Cauchy type 
and $\mathcal{D}^{y}_A(u ; q, t)$ is 
the Macdonald $q$-difference operator:
\begin{align}
&\mathcal{D}^{y}_A(u ; q, t) = 
\sum_{r=0}^n (-u)^r D_{A, r}^y (q,t),\\
&D_{A, r}^y (q,t)= t^{\binom{r}{2}} 
\sum_{\substack{K \subset \{1, \ldots, n \} \\ |K| = r}}
\prod_{\substack{k \in K \\ l \notin K}}
\frac{ty_k-y_l}{y_k -y_l} \prod_{k \in K}T_{q, y_k}.
\end{align}
The operator $\mathcal{H}_A^{x}(u;q, t)$ is defined by
\begin{align}
&\mathcal{H}_A^{x}(u;q, t) = \sum_{l=0}^\infty u^l H^x_{A,l},\\
&H^x_{A,l} = \sum_{\substack{\mu \in \mathbb{N}^m \\|\mu|=l}}
\prod_{1 \le i < j \le m}
\frac{q^{\mu_i}x_i-q^{\mu_j}x_j}{x_i-x_j}
\prod_{1 \le i ,j \le m}
\frac{(tx_i/x_j;q)_{\mu_i}}{(qx_i/x_j;q)_{\mu_i}} \prod_{1 \le i \le m} T_{q, x_i}^{\mu_i}.
\end{align}
We can obtain this fact as the special case of Kajihara's Euler transformation formula.
It is also known that the Macdonald polynomials $P_{A, \lambda}(x|q,t)$ for type A
are the joint eigenfunctions of $\mathcal{H}_A^{x}(u;q, t)$:
\begin{equation}
\mathcal{H}_A^{x}(u;q, t) P_{A,\lambda}(x|q, t)
=P_{A,\lambda}(x|q, t) \prod_{1 \le i \le m}
\frac{(ut^{m-i+1}q^{\lambda_i};q)_{\infty}}
{(ut^{m-i}q^{\lambda_i};q)_{\infty}} .
\end{equation}
The commutativity of this family $\{H^x_{A,l} \}_{l=0}^{\infty}$
is proved in \cite{Sano} through the Wronski relations in the elliptic setting.

In this section, we give the $BC$ type analogue of (\ref{eq:dual Cauchy type A}).
Namely, we construct an explicit operator $\mathcal{H}^x(u;q, t)$ which satisfies 
\begin{align}
\mathcal{H}^x(u;q, t) \Psi(x;y) = 
const. \cdot \widehat{\mathcal{D}}^y(u) \Psi(x;y), 
\label{eq:pre kernel identity of dual}\\
\widehat{\mathcal{D}}^y(u) = \sum_{r=0}^n (-1)^r e(u; \widehat{\alpha})_{q, n-r}
\widehat{D}_r^y. \label{eq:A^x and D^y}
\end{align}
Here $\Psi(x;y) := \prod_{\substack{1 \le i \le m \\1 \le k \le n}} e(x_i; y_k)$ 
is the kernel function of dual Cauchy type for type $BC$ introduced by Mimachi \cite{Mi2} and 
$\widehat{\ }$ means the operation of
replacing the parameters $(q, t)$ with $(t, q)$.
Therefore, $\widehat{\alpha} = \sqrt{q/t} \alpha$ and
$\widehat{D}_r^y$ are the $t$-difference operators $D_r^y(a, b, c, d|t, q)$.
Note that $\Psi(x; y)$ is expanded by the 
Koornwinder polynomials $P_\lambda(x)$ as follows \cite{Mi2}:
\begin{align}
\Psi(x; y) = \sum_{\lambda \subset (n^m)} (-1)^{\lambda^*}
P_\lambda(x) \widehat{P}_{\lambda^*} (y), \label{eq:expand by Koornwinder}
\\
\lambda^*= ( m - \lambda_n', m - \lambda_{n-1}', \ldots, m - \lambda_1'). 
\end{align}

\subsection{Affine Hecke algebras}
\label{subsec:Affine Hecke}
In order to guarantee the existence of $\mathcal{H}^x(u;q, t)$
satisfying (\ref{eq:pre kernel identity of dual}), 
we use the framework of affine Hecke algebras, due to
Cherednik \cite{Chere} and Macdonald \cite{Ma2}.
In this subsection, we recall Noumi's representations 
of affine Hecke algebras of type $C$ and the fundamental facts of
$q$-Dunkl operators.
Our notation is due to \cite{N1}.

We use the parameters $t_0, t_m, u_0 , u_m$ which are defined such that
\begin{equation}
(a, b, c, d) = (t_m^{\frac{1}{2}} u_m^{\frac{1}{2}}, -t_m^{\frac{1}{2}} u_m^{-\frac{1}{2}},
q^{\frac{1}{2}} t_0^{\frac{1}{2}} u_0^{\frac{1}{2}}, -q^{\frac{1}{2}} t_0^{\frac{1}{2}} u_0^{-\frac{1}{2}}) .
\end{equation}
We define the Lustig operators $T^x_0, T^x_1, \ldots, T^x_m$ as
\begin{align}
&T^x_0 = t_0^{\frac{1}{2}} +t_0^{-\frac{1}{2}}
\frac{(1-t_0^{\frac{1}{2}} u_0^{\frac{1}{2}} q^{\frac{1}{2}} x_1^{-1})
(1+t_0^{\frac{1}{2}} u_0^{-\frac{1}{2}} q^{\frac{1}{2}} x_1^{-1})}{1-q x_1^{-2}}(s^x_0 - 1) ,\\
&T^x_i = t^{\frac{1}{2}} +
t^{-\frac{1}{2}}\frac{1-t x_i/ x_{i+1}}{1-x_i/x_{i+1}}(s^x_i-1) \quad (i =1, \ldots, m-1) ,\\
&T^x_m = t_m^{\frac{1}{2}} +t_m^{-\frac{1}{2}}
\frac{(1-t_m^{\frac{1}{2}} u_m^{\frac{1}{2}} x_m)
(1+t_m^{\frac{1}{2}} u_m^{-\frac{1}{2}} x_m)}{1- x_m^2}(s^x_m - 1).
\end{align}
Here $s^x_0, s^x_1, \ldots, s^x_m$ are the simple reflections
\begin{align}
&(s^x_0 f) (x) = f(qx_1^{-1} , x_2, \ldots, x_m) ,\\
&(s^x_i f) (x) = f(x_1 , \ldots, x_{i+1}, x_i, \ldots, x_m)  \quad (i =1 ,\ldots, m-1) ,\\
&(s^x_m f) (x) = f(x_1 , x_2, \ldots, x_m^{-1}). 
\end{align}
Note that these operators $s^x_1, \ldots, s^x_m$ generate the Weyl group $W_m$ of type $BC_m$.
The algebra $\mathcal{H}(W_m^{\rm{aff}})$ generated by 
$T^x_0, T^x_1 ,\ldots , T^x_m$ is isomorphic to the affine Hecke algebra
of type $C_m$:
\begin{align}
&(T^x_{i}-t_i^{\frac{1}{2}})(T^x_{i}+t_i^{-\frac{1}{2}}) =0 \quad
(i=0, 1, \ldots, m), \\
&T^x_i T^x_{i+1} T^x_i = T^x_{i+1} T^x_i T^x_{i+1} \quad (i=1, \ldots, m-2),\\
&T^x_{i} T^x_{i+1} T^x_{i} T^x_{i+1} = T^x_{i+1} T^x_{i} T^x_{i+1} T^x_i
\quad (i=0, m-1),\\
&T_i^x T_j^x = T_j^x T_i^x \quad ( |i- j| \ge 2).
\end{align}
Here we wrote $t_1 = \cdots = t_{m-1}=t$.
The $q$-Dunkl operators $Y^x_1, \ldots, Y^x_m$ are defined by
\begin{equation}
Y^x_i = (T^x_i \cdots T^x_{m-1}) (T^x_m \cdots T^x_0) ((T^x_1)^{-1} \cdots (T^x_{i-1})^{-1}) \ (i=1, \ldots, m).
\end{equation}
We denote by $\mathbb{C}(x)[T_{q, x}^{\pm1}]$
the ring of $q$-difference operators with rational coefficients. 
For any $A^x \in \mathcal{H}(W_m^{\rm{aff}})$,
$A^x$ is expressed as $\sum_{w \in W_m} A^x_w w\ (A^x_w \in \mathbb{C}(x)[T_{q, x}^{\pm1}])$.
Then we define the $q$-difference operator $L^x_A$ by
$L^x_A =\sum_{w \in W_m} A^x_w$.
It is known that the following fact holds.

For any $W_m$-invariant Laurent polynomial $f(\xi)$
in the variables $\xi=(\xi_1, \ldots, \xi_m)$, and 
for any $W_m$-invariant Laurent
polynomial $\varphi(x) \in \mathbb{C}[x^{\pm1}]^{W_m}$,
one has 
\begin{equation}
f(Y^x) \varphi(x) = L^x_{f(Y^x)} \varphi(x).
\end{equation}
Furthermore, the $q$-difference operator $L^x_f:=L^x_{f(Y^x)}$ satisfies 
for any partition $\lambda$
\begin{equation}
f(Y^x)P_\lambda(x) = L^x_f P_\lambda(x) = P_\lambda(x) f(\alpha  t^{\rho_m}
q^\lambda) . \label{eq:L diff eqs}
\end{equation}
In particular, the $q$-difference operators $L^x_{f}$ for
the interpolation polynomials $f=e_r(\xi;\alpha|t)$ 
of column type give rise to van Diejen's operators $D^x_r$.
From this view point, we call $D^x_r$ ``column type'' $q$-difference operators.
Since $\{e_r(\xi;\alpha|t) \}_{r=1}^{m}$ is the generator system of the ring 
$\mathbb{C}[\xi^{\pm1}]^{W_m}$ 
of $W_m$-invariant Laurent polynomials in $m$ variables,
$L^x_f$ is an element of $\mathbb{C}[D^x_1, \ldots, D^x_m]$
for any $f(\xi) \in \mathbb{C}[\xi^{\pm1}]^{W_m}$.
The operator $\mathcal{H}^x(u; q,t)$ to be constructed in the next subsection
is a generating function of ``row type'' $q$-difference operators.

\subsection{Construction of row type $q$-difference operators}
\label{subsec:Construction}
The results of this subsection are based on a discussion with M. Noumi.

Let $\xi=(\xi_1, \ldots, \xi_m)$ and $\eta =(\eta_1, \ldots, \eta_n)$.
We define
\begin{align}
H(u; \xi) =\Phi_0(u; \xi|q, t) =u^{m \gamma}\prod_{i=1}^m
\frac{(\sqrt{tq} u \xi_i^{\pm1}; q)_{\infty}}
{(\sqrt{q/t} u\xi_i^{\pm1}; q)_{\infty}},\quad
E(u; \eta) = \prod_{k=1}^n e(u; \eta_k),
\end{align}
where $\gamma$ is a complex number such that $q^{\gamma}=t$.
Note that $E(u;\eta)$ is a generating function of $e_r(\eta;\alpha|t)$:
\begin{equation}
E(u; \eta) = \sum_{r=0}^n (-1)^r e_r(\eta; \alpha|t) e(u;\alpha)_{t, n-r}.
\end{equation}
This implies that 
\begin{align}
E(u; Y^y) P_\lambda(y) = P_\lambda(y) \prod_{i=1}^n
e(u; \alpha t^{n-i}q^{\lambda_i}).
\end{align}
Namely, the operator $L^y_f$ for $f=E(u;\eta)$ is 
the generating function $\mathcal{D}^y(u)$ of column type operators $D_r^y$.
In the following, we regard $H(u; \xi)$ as  an element of
$u^{-m\gamma}\mathbb{C}[\xi^{\pm1}]^{W_m} [[u]]$.
Namely, $u^{m\gamma} H(u; \xi)$ is a formal power series in $u$
with coefficients in the ring of $W_m$-invariant Laurent polynomials in $\xi$.
\begin{lemm}\label{Lemm:dual Cauchy}
The operators $H(u;Y^x)$ and $E(u; Y^y)$ satisfy the following identity
as formal power series in $u$:
\begin{equation}
\frac{H(u; Y^x)}{H(u; \alpha t^{\rho_m})} \Psi(x;y)
= \frac{e(u; \widehat{\alpha})_{n}}
{e(u; \widehat{\alpha}t^m )_{n}}
\frac{E(u; \widehat{Y}^y)}
{ E(u; \widehat{\alpha}q^{\rho_n})} \Psi(x;y) \\
= \frac{E(u; \widehat{Y}^y)}
{ e(u; \widehat{\alpha}t^m)_{n}} \Psi(x;y). \label{eq:Lemm dual Cauchy}
\end{equation}
\end{lemm}
\begin{proof}
Note first that $E(u;\widehat{\alpha}q^{\rho_n})
= e(u;\widehat{\alpha})_n$.
From (\ref{eq:expand by Koornwinder}) and
(\ref{eq:L diff eqs}),
the formula (\ref{eq:Lemm dual Cauchy}) is equivalent to the identity
on the eigenvalue:
\begin{align}
\frac{H(u; \alpha t^{\rho_m}q^{\lambda})}{H(u; \alpha t^{\rho_m})} =
\frac{e(u; \widehat{\alpha})_{n}}
{e(u;\widehat{\alpha} t^m )_{n}}
\frac{E(u; \widehat{\alpha} q^{\rho_n}t^{\lambda^*})}{E(u; \widehat{\alpha}q^{\rho_n})} \quad (\lambda \subset (n^m)). \label{eq:eigenvalue of dual cauchy}
\end{align}
The left-hand side is equal to
\begin{equation}
\prod_{1 \le i \le m} \frac{(\sqrt{q/t} u \alpha t^{m-i}; q)_{\lambda_i}
(\sqrt{tq} u \alpha^{-1} t^{-m+i} q^{-\lambda_i}; q)_{\lambda_i}}
{(\sqrt{tq} u \alpha t^{m-i}; q)_{\lambda_i}
(\sqrt{q/t} u \alpha^{-1} t^{-m+i} q^{-\lambda_i}; q)_{\lambda_i}}
\label{eq:intertwine left}.
\end{equation}
Using the notation $c_{ij} = \alpha t^{m-i} q^{j-1}\ (1\le i \le m, 1 \le j\le n)$,
we obtain
\begin{align}
(\ref{eq:intertwine left})=\prod_{(i, j) \in \lambda} \frac{t(1-\sqrt{q/t} c_{ij} u)(1-\sqrt{q/t} c_{ij} u^{-1})}
{(1-\sqrt{tq} c_{ij} u)(1-\sqrt{tq} c_{ij} u^{-1})}
= \prod_{(i, j) \in \mu} \frac{\langle u^{\pm1} \sqrt{q/t}c_{ij}\rangle}
{\langle u^{\pm1} \sqrt{tq} c_{ij}\rangle} .
\end{align}
On the other hand, since we compute
\begin{align}
\frac{E(u; \widehat{\alpha} q^{\rho_n}t^{\lambda^*})}{E(u; \widehat{\alpha}q^{\rho_n})}
&= \prod_{k=1}^n \frac{\langle \sqrt{q/t}u \alpha  q^{n-k}t^{\lambda_k^*}\rangle
\langle \sqrt{t/q}u \alpha^{-1}  q^{k-n}t^{-\lambda_k^*}\rangle}
{\langle \sqrt{q/t} u\alpha q^{n-k}\rangle \langle \sqrt{t/q}u\alpha^{-1} q^{k-n}\rangle} \nonumber \\
&= \prod_{k=1}^n \frac{\langle \sqrt{tq}u \alpha  q^{n-k}t^{\lambda_k^*-1}\rangle
\langle \sqrt{tq}u^{-1} \alpha  q^{n-k}t^{\lambda_k^*-1}\rangle}
{\langle \sqrt{q/t}u\alpha  q^{n-k}\rangle \langle \sqrt{q/t}u^{-1} \alpha  q^{n-k}\rangle} \nonumber \\
&= \prod_{(i, j) \in (n^m)\setminus \lambda } \frac{\langle u^{\pm1} \sqrt{tq}c_{ij}\rangle}
{\langle u^{\pm1} \sqrt{q/t} c_{ij}\rangle},
\end{align}
and hence
the right-hand side of (\ref{eq:eigenvalue of dual cauchy}) is equal to
\begin{align}
\frac{\langle \sqrt{q/t} \alpha u^{\pm1} \rangle_{q,n}}{\langle \sqrt{tq}  \alpha t^{m-1} u^{\pm1} \rangle_{q,n}}
\prod_{(i, j) \in (n^m)\setminus \lambda } \frac{\langle u^{\pm1} \sqrt{tq}c_{ij}\rangle}
{\langle u^{\pm1} \sqrt{q/t} c_{ij}\rangle} 
&= \prod_{(i, j) \in (n^m)} \frac{\langle u^{\pm1} \sqrt{q/t}c_{ij}\rangle}
{\langle u^{\pm1} \sqrt{tq} c_{ij}\rangle}
\prod_{(i, j) \in (n^m)\setminus \lambda } \frac{\langle u^{\pm1} \sqrt{tq}c_{ij}\rangle}
{\langle u^{\pm1} \sqrt{q/t} c_{ij}\rangle} \nonumber \\
&= \prod_{(i, j) \in \lambda} \frac{\langle u^{\pm1} \sqrt{q/t}c_{ij}\rangle}
{\langle u^{\pm1} \sqrt{tq} c_{ij}\rangle}.
\end{align} 
\end{proof}
Next we show that $\frac{H(u; \xi)}{H(u; \alpha t^{\rho_m})}$
is a generating function of the row type
interpolation polynomials $h_{l}(\xi; \alpha| q, t)$
introduced by \cite{KNS}:
\begin{equation}
h_{l}(\xi; \alpha| q, t) = \sum_{\substack{\nu \in \mathbb{N}^m \\
|\nu| = l}}
\frac{\langle t \rangle_{\nu_1} \cdots \langle t \rangle_{\nu_m}}
{\langle q \rangle_{\nu_1} \cdots \langle q \rangle_{\nu_m}}
e(\xi_1; \alpha)_{\nu_1} e(\xi_2; tq^{\nu_1} \alpha)_{\nu_2}
\cdots e(\xi_m; t^{m-1} q^{\nu_1 + \cdots \nu_{m-1}} \alpha)_{\nu_m}.
\end{equation}
Note that the Laurent polynomial $h_l (\xi; \alpha|q, t)$
is $W_m$-invariant and satisfies the following interpolation
property:
For any partition $\mu \not\supset (l)$,
\begin{equation}
h_l(\alpha t^{\rho_m} q^{\mu}; \alpha|q, t) =0 .
\end{equation}
\begin{lemm}
The following identity holds as formal power series in $u$:
\begin{equation}
\frac{H(u; \xi)}{H(u; \alpha t^{\rho_m})} =
\sum_{l=0}^\infty \frac{h_l(\xi; \alpha|q, t)}{e(u; t^m\sqrt{q/t} \alpha)_{l}}.
\label{eq:gene func of row type}
\end{equation}
\end{lemm}
\begin{proof}
If $t=q^{-k}\ (k=0, 1, 2, \ldots)$,
from Lemma 5.4 in \cite{KNS} one has
\begin{equation}
H(u; \xi)= \prod_{1 \le i \le m} e(u; q^{\frac{1}{2}(1-k)}\xi_i )_{k}
=\sum_{l=0}^{km} h_l(\xi; \alpha|q, t) e(u; \sqrt{tq}/\alpha)_{km-l}
\label{eq:lemma 5.4}.
\end{equation}
Since $H(u;\alpha t^{\rho_m})$ with $t=q^{-k}$
equals $e(u; \sqrt{tq}/\alpha)_{km}$,
by dividing the both sides of (\ref{eq:lemma 5.4}) by 
$e(u;\sqrt{tq}/\alpha)_{km}$,
we obtain this lemma in the case of $t=q^{-k}$.
In the formal power series of $u$ in each side of 
(\ref{eq:gene func of row type}), all the coefficients are the rational
functions in $t^{\frac{1}{2}}$. Hence the identity 
(\ref{eq:gene func of row type}) follows from its validity at
infinitely many values of $t=q^{-k}\ (k=0, 1, 2, \ldots)$. 
\end{proof}
From this lemma, it follows that
\begin{equation}
\frac{H(u; Y^x)}{H(u; \alpha t^{\rho_m})}
= \sum_{l=0}^{\infty} \frac{h_l(Y^x; \alpha|q,t)}{e(u; t^m\sqrt{q/t} \alpha)_{ l}}.
\end{equation}
We now define the $q$-difference operators 
$H^x_{l}:=H^x_l(a, b ,c,d|q,t)\ (l=0, 1, 2, \ldots)$ 
to be $L^x_{f}$ for $f= h_l(\xi; \alpha|q, t)$, so that
\begin{align}
H_l^x P_\lambda(x) = P_\lambda(x) h_l(\alpha t^{\rho_m}q^\lambda; \alpha|q, t). \label{eq:row type diff eq.}
\end{align}
We call these operators $H^x_l\ (l=0, 1, 2, \ldots)$
``row type'' $q$-difference operators.
We also introduce a generating function $
\mathcal{H}^x(u) := \mathcal{H}^x(u;q, t)$
of $H_l^x$ by
\begin{equation}
\mathcal{H}^x(u;q, t)=\sum_{l=0}^{\infty}
\frac{H_l^x}{e(u;t^m \sqrt{q/t}\alpha)_l}  \in 
\mathbb{C}(x)[T_{q, x}^{\pm1}][[u]].
\end{equation}
From Lemma \ref{Lemm:dual Cauchy} , we obtain a
``kernel identity of dual Cauchy type''.
\begin{theo}\label{Theo:dual Cauchy}
The kernel function of dual Cauchy type intertwines the $q$-difference operator
$\mathcal{H}^x(u)$ with the $t$-difference operator 
$\widehat{\mathcal{D}}^y(u)$:
\begin{equation}
\mathcal{H}^x(u) \Psi(x; y) = 
\frac{\widehat{\mathcal{D}}^y(u) }{e(u; \widehat{\alpha} t^m)_{n}}
\Psi(x;y). \label{eq:Theo dual Cauchy}
\end{equation}
\end{theo}
Theorem \ref{Theo:dual Cauchy} gives the relationship between
$H_l$ and van Diejen's operators $D_r$. 
\begin{theo}
For any integer $l=0, 1, 2, \ldots, n$,
the following equation holds:
\begin{equation}
(-1)^l H_l^x \Psi(x; y) = \sum_{0 \le s \le l} \frac{\langle q^{n-l+1} \rangle_{s}}{\langle q \rangle_{s}}
\langle q^{1-l} t^{-m}, t^{m-1} q^n \alpha^2 \rangle_{s} \widehat{D}_{l-s}^y \Psi(x;y). \label{eq:rela of H and D}
\end{equation}
\end{theo}
\begin{proof}
Since $H_l^x P_\mu(x) =0$ if $\mu \subset (n^m)$ and $l >n$,
\begin{align}
\mathcal{H}^x(u) \Psi(x; y)
=\left(1 + \frac{1}{e(u; t^m \sqrt{q/t} \alpha)} H_1^x  +\cdots +
\frac{1}{e(u; t^m \sqrt{q/t} \alpha)_{n}} H_n^x \right) \Psi(x;y).
\label{eq:H Phi}
\end{align}
By using the $q$-$\text{Saalsch}\ddot{\text{u}}\text{tz}$ sum (\ref{eq:Saal}),
the right-hand side of (\ref{eq:Theo dual Cauchy}) is expressed by
\begin{align}
\frac{\widehat{D}_y(u)}{e(u; t^m\widehat{\alpha})_{n}} \Psi(x;y)
&= \sum_{l=0}^n \frac{1}{e(u; t^m \sqrt{q/t} \alpha)_{l}} 
\sum_{0 \le r \le l} (-1)^l
\begin{bmatrix} 
n-r  \\ 
l-r  
\end{bmatrix}_{q}
\langle t^{-m}q^{1-l} , t^{m-1}q^n \alpha^2 \rangle_{l-r} \widehat{D}_r^y \Psi(x;y).
\label{eq:D Phi}
\end{align}
Comparing the coefficient of $\frac{1}{e(u; t^m\sqrt{q/t}\alpha)_l}$
in (\ref{eq:H Phi}) with that in (\ref{eq:D Phi}) for each $l$,
we obtain (\ref{eq:rela of H and D}). 
\end{proof}
\subsection{Explicit formulas of $H_l$}
In the previous subsection, we defined the row type $q$-difference operators $H^x_l$ by $q$-Dunkl operators
and showed the relationship between $H^x_l$ and $\widehat{D}_r^y$.
However it is difficult to compute the explicit expressions of operators
$H_l^x$ by means of the $q$-Dunkl operators.
In this subsection, by using the special case of Theorem \ref{Theo:trans form BC}, we give the explicit formulas of $H^x_l$.

For $\nu \in \mathbb{N}^m$ with $0 \le |\nu| \le l$, we define
\begin{align}
&H^{(l)}_{\nu}(x)
=\sum_{\substack{\nu \le \nu^{+} \\ |\nu| \le |\nu^+| \le l}}
\Biggl(
\sum_{\substack{0 \le \nu^- \le (l^m)-\nu^+ \\ |\nu^-|=l-|\nu^+|}} \Biggl( (-1)^l 
\prod_{1 \le i \le m} \frac{\langle ax_i, bx_i, cx_i, dx_i \rangle_{\nu_i^+}
\langle a/x_i, b/x_i, c/x_i, d/x_i \rangle_{\nu_i^-}}
{\langle x_i^2 \rangle_{\nu_i + \nu_i^+} \langle x_i^{-2} \rangle_{\nu_i + \nu_i^-}}
\nonumber \\
&\cdot 
\prod_{1 \le i\le j \le m}
\frac{\langle q^{\nu_i+\nu_j} x_ix_j \rangle}
{\langle x_ix_j  \rangle}
\prod_{1 \le i<j \le m} 
\frac{\langle q^{\nu_i-\nu_j} x_i/x_j, q^{\nu_i^+ - \nu_j^+} x_i/x_j,
q^{\nu_j^- -\nu_i^-} x_i/x_j \rangle
}
{\langle x_i/x_j, x_i/x_j, x_i/x_j \rangle }
\nonumber \\
&\cdot 
\prod_{1 \le i<j \le m} 
\frac{\langle tx_i x_j, qx_ix_j \rangle_{\nu_i^+ + \nu_j^+}
\langle tx_i^{-1} x_j^{-1}, qx_i^{-1} x_j^{-1} \rangle_{\nu_i^- + \nu_j^-}}
{\langle x_i x_j\rangle_{\nu_i + \nu_j^+}
\langle x_i x_j\rangle_{\nu_i^+ + \nu_j}
\langle x_i^{-1} x_j^{-1}\rangle_{\nu_i + \nu_j^-}
\langle x_i^{-1} x_j^{-1}\rangle_{\nu_i^- + \nu_j}}
\prod_{1 \le i, j\le m} 
\frac{\langle q^{\nu_i^+ -\nu_j^-}x_ix_j,  x_i^{-1}x_j^{-1} \rangle
}
{\langle q^{\nu_i+\nu_j^+} x_ix_j, q^{\nu_i+\nu_j^-} x_i^{-1} x_j^{-1} \rangle}
\nonumber \\
&\cdot 
\prod_{1 \le i, j \le m}
\frac{\langle q^{-\nu_j^+} x_i/x_j \rangle_{\nu_i}
\langle tx_i/x_j\rangle_{\nu_i^+}
\langle q^{\nu_j^+ +1}x_j/x_i, tq^{\nu_j^+}x_j/x_i\rangle_{\nu_i^-}
\langle x_i x_j, q^{\nu_j^-+1} x_i^{-1} x_j^{-1}\rangle_{\nu_i}}
{\langle q x_i/x_j \rangle_{\nu_i} \langle qx_i/x_j \rangle_{\nu_i^+}
\langle qx_j/x_i , q^{\nu_j+1}x_j/x_i \rangle_{\nu_i^-}
\langle qx_i x_j\rangle_{\nu^+_i} \langle q^{1-\nu_j} x^{-1}_i x^{-1}_j\rangle_{\nu^-_i}}
\Biggr) \Biggr) 
. \label{eq:explicit of H}
\end{align}
In particular,
for $|\nu|=l, \nu \in \mathbb{N}^m$ we have
\begin{align}
H^{(l)}_{\nu}(x)
&=
\prod_{1 \le i \le m} \frac{\langle ax_i, bx_i, cx_i, dx_i \rangle_{\nu_i}}{\langle x_i^2 \rangle_{2\nu_i}} 
 \prod_{1 \le i<j \le m}
\frac{\langle tx_i x_j \rangle_{\nu_i + \nu_j} \langle q^{\nu_i-\nu_j} x_i/x_j \rangle}
{\langle x_i x_j \rangle_{\nu_i + \nu_j} \langle x_i/x_j \rangle}
\prod_{1 \le i, j \le m}
\frac{\langle tx_i/x_j \rangle_{\nu_i}}
{\langle qx_i/x_j \rangle_{\nu_i}}.
\end{align}
For any $\nu =(\epsilon_1 \nu'_1, \ldots, \epsilon_m \nu'_m)\in \mathbb{Z}^m\ (\epsilon_i=\pm1, \nu_i' \in \mathbb{N})$ such that
$\sum_{i=1}^m \nu_i' \le l$,
we write $|\nu| =\sum_{i=1}^m \nu'_i$ and set
$H^{(l)}_\nu (x;a, b,c, d) =H^{(l)}_{\nu'}(x_1^{\epsilon_1},\ldots,x_m^{\epsilon_m};a,b,c,d)$.
\begin{theo}\label{Theo:explicit of H_l}
The row type $q$-difference operators $H_l^x\ (l=0 ,1, 2, \ldots)$ 
are expressed explicitly as
\begin{align}
&H_l^x = \sum_{\substack{\nu \in \mathbb{Z}^m \\ 0 \le |\nu| \le l}}
H^{(l)}_{\nu}(x; a, b, c, d)
\prod_{1 \le i \le m} T^{\nu_i}_{q, x_i}.
\end{align}
Namely, the Koornwinder polynomials $P_\lambda(x)$ are the joint eigenfunctions of $H^x_l\ (l=0 ,1, 2, \ldots)$:
\begin{equation}
H_l^x P_\lambda(x) = P_\lambda(x) h_l(\alpha t^{\rho_m}q^\lambda; \alpha|q, t).  \label{eq:q-diff eq of H_l}
\end{equation}
\end{theo}
\begin{proof}
We consider Theorem \ref{Theo:trans form BC} 
in the case of $\alpha = (M, \ldots, M) \in \mathbb{N}^m,
\beta=(1, \ldots, 1) \in \mathbb{N}^n$:
\begin{align}
&\prod_{1 \le i \le m} \frac{\langle tq x_i/a_1, tq x_i/a_2, tq x_i/a_3, tq x_i/a_4 \rangle_{M}}
{\langle tq x_i^2, t^2q x_i^2\rangle_{M}}
\prod_{1 \le i < j \le m}
\frac{\langle t q^{M+1}x_i x_j,  t^2 q^{M+1}x_i x_j\rangle_{M}} 
{\langle tqx_i x_j,  t^2qx_i x_j\rangle_{M}} \nonumber \\
&\cdot
\sum_{0 \le \mu^{-} \le \nu^{-} \le \nu^{+} \le \mu^{+} \le (M^m)}
\Biggl( (-1)^{mM+|\nu^+| + |\nu^-|}
e(u; \sqrt{tq}/a_0)_{|\nu^+|-|\nu^-|} \nonumber \\
&\cdot \prod_{1 \le i \le m} 
\frac{\langle a_1 x_i, a_2 x_i, a_3 x_i, a_4 x_i\rangle_{\nu_i^{-}}}
{\langle  tq x_i/a_1, tq x_i/a_2, tq x_i/a_3, tq x_i/a_4 \rangle_{\nu_i^{+}}} 
\prod_{1 \le i \le j \le m}
\frac{\langle q^{\mu_i^- + \mu_j^-} x_i x_j,
t^2 q^{\mu_i^+ + \mu_j^+} x_i x_j\rangle}
{\langle x_i x_j, t^2 x_i x_j \rangle}\nonumber \\
&\cdot
\prod_{1 \le i < j \le m}
\frac{\langle q^{\mu_i^- - \mu_j^-} x_i/x_j,
q^{\mu_i^+ - \mu_j^+} x_i/x_j,
q^{\nu_i^- - \nu_j^-} x_i/x_j,
q^{\nu_i^+ - \nu_j^+} x_i/x_j\rangle
 \langle tx_i x_j, qx_i x_j\rangle_{\nu_i^- + \nu_j^-}}
{\langle x_i/x_j, x_i/x_j, x_i/x_j, x_i/x_j\rangle
\langle  tqx_i x_j, t^2 x_i x_j \rangle_{\nu_i^+ + \nu_j^+}}
\nonumber \\
&\cdot \prod_{1 \le i, j \le m}
\frac{\langle tq^{\mu_i^- + \mu_j^+} x_i x_j,
tq^{\nu_i^- + \nu_j^+} x_i x_j,
q^{\mu_i^- - \mu_j^+} x_i/tx_j\rangle
\langle  tx_i x_j \rangle_{\mu_i^- + \nu_j^+}
\langle  t^2x_i x_j \rangle_{\mu_i^+ + \nu_j^+}}
{\langle tx_i x_j,tx_i x_j, x_i/tq^{\mu_j^+}x_j\rangle
\langle  qx_i x_j \rangle_{\mu_i^-+ \nu_j^-}
\langle  tqx_i x_j \rangle_{\mu_i^+ + \nu_j^-}} \nonumber \\
&\cdot \prod_{1 \le i, j \le m}
\frac{\langle  x_i x_j, x_i/q^{\nu_j^-}x_j , x_i/tq^{M}x_j  \rangle_{\mu_i^-}
\langle  tx_i x_j, x_i/q^{M}x_j \rangle_{\mu_i^+ }
\langle  qx_i/tq^{\nu_j^+}x_j , x_i/q^{\nu_j^+}x_j \rangle_{\nu_i^-}
}
{\langle  tq^{M+1}x_i x_j, qx_i/tq^{\nu_j^+}x_j, qx_i/ x_j \rangle_{\mu_i^-}
\langle  t^2 q^{M+1} x_i x_j,  qx_i/ x_j  \rangle_{\mu_i^+}
\langle qx_i/tq^{\mu_j^+}x_j, qx_i/ x_j \rangle_{\nu_i^-}
} \nonumber \\
&\cdot 
\prod_{1 \le i, j \le m}
\frac{\langle  x_i/q^{\mu_j^+}x_j \rangle_{\nu_i^+}}
{\langle  qx_i/x_j \rangle_{\nu_i^+}}
\prod_{\substack{1\le i \le m\\ 1 \le k \le n}}
\frac{\langle q x_i y_k, q x_i/y_k \rangle_{\mu_i^-}
\langle tq^{\mu_i^+} x_i y_k,  tq^{\mu_i^+} x_i/y_k \rangle_{M-\mu_i^+}}
{\langle x_i y_k,  x_i/y_k \rangle_{\mu_i^-}
\langle tq^{\mu_i^+ +1} x_i y_k,  tq^{\mu_i^+ +1} x_i/y_k \rangle_{M-\mu_i^+}}\Biggr)\nonumber \\
&=e(u; \sqrt{tq}/a_0)_{mM-n} 
\prod_{1 \le k \le n} 
\frac{\langle a_1y_k, a_2y_k, a_3y_k, a_4y_k \rangle}
{\langle y_k^2, ty_k^2\rangle}
\prod_{1 \le k < l \le n}
\frac{\langle q y_k y_l,  tqy_k y_l\rangle}
{\langle y_k y_l,  ty_k y_l\rangle} \nonumber \\
&\cdot
\sum_{0 \le \lambda^{-} \le \kappa^{-} \le \kappa^{+} \le \lambda^{+} \le (1^n)}
(-1)^{n+|\kappa^+| + |\kappa^-|}
\Biggl(e(u; \sqrt{q/t} a_0)_{|\kappa^+|-|\kappa^-|}  \nonumber \\
&\cdot
\prod_{1 \le k \le n} \frac{\langle y_k/a_1, y_k/a_2,
y_k/a_3, y_k/a_4 \rangle_{\kappa_k^{-}}} 
{\langle a_1y_k, a_2y_k, a_3y_k, a_4y_k
\rangle_{\kappa_k^{+}}} 
\prod_{1 \le k \le l \le n}
\frac{\langle q^{\lambda_k^- + \lambda_l^- } y_k y_l/tq,
tq^{\lambda_k^+ + \lambda_l^+ } y_k y_l/q \rangle}
{\langle y_k y_l/tq,  t y_k y_l/q\rangle}\nonumber \\
&\cdot
\prod_{1 \le k < l \le n}
\frac{\langle q^{\lambda_k^- - \lambda_l^-} y_k/y_l,
q^{\lambda_k^+ - \lambda_l^+} y_k/y_l,
q^{\kappa_k^- - \kappa_l^-} y_k/y_l,
q^{\kappa_k^+ - \kappa_l^+} y_k/y_l\rangle
\langle  y_k y_l/q, y_k y_l/t \rangle_{\kappa_k^- + \kappa_l^-}}
{\langle y_k/y_l,  y_k/y_l, y_k/y_l, y_k/y_l\rangle
\langle ty_k y_l/q, y_k y_l \rangle_{\kappa_k^+ + \kappa_l^+}}\nonumber \\
&\cdot \prod_{1 \le k, l \le n}
\frac{\langle q^{\lambda_k^- + \lambda_l^+} y_k y_l/q,
q^{\kappa_k^- + \kappa_l^+} y_k y_l/q,
 q^{\lambda_k^- - \lambda_l^+} y_k/ty_l\rangle
 \langle  y_k y_l/q \rangle_{\lambda_k^- + \kappa_l^+}
 \langle  ty_k y_l/q \rangle_{\lambda_k^+ + \kappa_l^+}}
 {\langle y_k y_l/q, y_k y_l/q, y_k/tq^{\lambda_l^+}y_l \rangle
 \langle  y_k y_l/t \rangle_{\lambda_k^-+ \kappa_l^-}
 \langle  y_k y_l \rangle_{\lambda_k^+ + \kappa_l^-}}\nonumber \\
&\cdot \prod_{1 \le k, l \le n}
\frac{\langle  y_k y_l/tq, y_k/q^{\kappa_l^-}y_l , y_k/tqy_l \rangle_{\lambda_k^-}
\langle  y_k y_l/q , y_k/qy_l \rangle_{\lambda_k^+}
\langle  qy_k/tq^{\kappa_l^+}y_l , y_k/q^{\kappa_l^+}y_l \rangle_{\kappa_k^-}
\langle  y_k/q^{\lambda_l^+}y_l \rangle_{\kappa_k^+}}
{\langle  q y_k y_l, qy_k/tq^{\kappa_l^+}y_l, qy_k/ y_l  \rangle_{\lambda_k^-}
\langle  tq y_k y_l, qy_k/ y_l \rangle_{\lambda_k^+}
\langle qy_k/tq^{\lambda_l^+}y_l, qy_k/ y_l \rangle_{\kappa_k^-}
\langle qy_k/ y_l \rangle_{\kappa_k^+}}\nonumber \\
&\cdot \prod_{\substack{ 1 \le k \le n \\ 1\le i \le m}}
\frac{\langle q^{M} y_k x_i ,  y_k/tx_i \rangle_{\lambda_k^-}
\langle tq^{\lambda_k^+}y_k x_i,  q^{\lambda_k^+-M}y_k/x_i \rangle_{1-\lambda_k^+}}
{\langle y_kx_i , y_k/tq^{M}x_i \rangle_{\lambda_k^-}
\langle tq^{M+\lambda_k^+}  y_kx_i, q^{\lambda_k^+}y_k/x_i \rangle_{1-\lambda_k^+}}\Biggr) .
\label{eq:special of BC duality}
\end{align}
Then the both sides are $W_m$-invariant for the variables $x$ and
$W_n$-invariant for $y$.
Since 
\begin{align}
&\prod_{\substack{1\le i \le m\\ 1 \le k \le n}}
\frac{\langle q x_i y_k, q x_i/y_k \rangle_{\mu_i^-}
\langle tq^{\mu_i^+}x_i y_k,  tq^{\mu_i^+} x_i/y_k \rangle_{M-\mu_i^+}}
{\langle x_i y_k, x_i/y_k \rangle_{\mu_i^-}
\langle t q^{1+\mu_i^+} x_i y_k,  tq^{\mu_i^+ +1} x_i/y_k \rangle_{M-\mu_i^+}} \nonumber \\
&= \prod_{\substack{1\le i \le m\\ 1 \le k \le n}}
\frac{T_{q, x_i}^{\mu_i^-}  \langle  x_i y_k,  x_i/y_k \rangle}
{\langle x_i y_k, x_i/y_k \rangle}
\prod_{\substack{1\le i \le m\\ 1 \le k \le n}}
 \frac{T_{q, x_i}^{-(M-\mu_i^+)} \langle tq^{M}x_i y_k,  tq^{M} x_i/y_k \rangle}
{\langle tq^{M}x_i y_k,  tq^{M} x_i/y_k \rangle}, \label{eq:bilin1}
\end{align}
the left-hand side of (\ref{eq:special of BC duality}) is expressed as 
the bilinear form:
\begin{equation}
\sum_{0 \le \mu^- \le \mu^+ \le M} A_{\mu^-, \mu^+}(u; x)  
\prod_{\substack{1\le i \le m\\ 1 \le k \le n}}
\frac{T_{q, x_i}^{\mu_i^-} \langle  x_i y_k,  x_i/y_k \rangle}
{\langle x_i y_k, x_i/y_k \rangle}
\prod_{\substack{1\le i \le m\\ 1 \le k \le n}}
 \frac{T_{q, x_i}^{-(M-\mu_i^+)} \langle tq^{M}x_i y_k,  tq^{M} x_i/y_k \rangle}
{\langle tq^{M}x_i y_k,  tq^{M} x_i/y_k \rangle}. \label{eq:bilin2}
\end{equation}
In terms of this form, Theorem \ref{Theo:trans form BC} is regarded as 
a kind of bilinear transformation formula.

We consider the special case $t=q^{-M}$ in (\ref{eq:special of BC duality}).
Then the factors involving both variables $x$ and $y$ are expressed by
\begin{align}
&\prod_{\substack{1\le i \le m\\ 1 \le k \le n}}
\frac{\langle q x_i y_k, q x_i/y_k \rangle_{\mu_i^-}
\langle tq^{\mu_i^+}x_i y_k,  tq^{\mu_i^+} x_i/y_k \rangle_{M-\mu_i^+}}
{\langle x_i y_k, x_i/y_k \rangle_{\mu_i^-}
\langle t q^{1+\mu_i^+} x_i y_k,  tq^{\mu_i^+ +1} x_i/y_k \rangle_{M-\mu_i^+}} \nonumber \\ 
&=\prod_{\substack{1\le i \le m}}\left[
\frac{T_{q, x_i}^{\mu_i^-} \Psi(x_i;y)}
{\Psi(x_i;y)}
\frac{T_{q, x_i}^{-(M-\mu_i^+)} \Psi(x_i;y)}
{\Psi(x_i;y)}\right],\\
&\prod_{\substack{ 1 \le k \le n \\ 1\le i \le m}}
\frac{\langle q^{M} y_k x_i ,  y_k/tx_i \rangle_{\lambda_k^-}
\langle tq^{\lambda_k^+}y_k x_i,  q^{\lambda_k^+-M}y_k/x_i \rangle_{1-\lambda_k^+}}
{\langle y_kx_i , y_k/tq^{M}x_i \rangle_{\lambda_k^-}
\langle tq^{M+\lambda_k^+}  y_kx_i, q^{\lambda_k^+}y_k/x_i \rangle_{1-\lambda_k^+}} \nonumber \\
&=\prod_{\substack{1 \le k \le n}}
\left[\frac{T_{t, y_k}^{1- \lambda_k^+} \Psi(x;y_k)}
{\Psi(x;y_k)}
\frac{T_{t, y_k}^{-\lambda_k^-}\Psi(x;y_k) }
{\Psi(x;y_k)}\right].
\end{align}
We check that each side of (\ref{eq:special of BC duality}) 
does not have a pole at the point $t=q^{-M}$.
We have only to examine the following factor in the left-hand side:
\begin{align}
\prod_{1 \le i \le m}
\frac{\langle q^{\mu_i^-  - \mu_i^+}/t \rangle
\langle q^{-M}/t \rangle_{\mu_i^{-}}
\langle q^{1-\nu_i^+}/t \rangle_{\nu_i^{-}}}
{\langle q^{-\mu_i^+}/t \rangle
\langle q^{1-\nu_i^+}/t \rangle_{\mu_i^-}
\langle q^{1-\mu_i^+}/t \rangle_{\nu_i^-}}. 
\label{eq:check l.h.s}
\end{align}
If $\mu_i^+ = M$ for some $i=1, \ldots, m$,
the denominator has a zero at $t=q^{-M}$.
But when $\mu_i^- =0$, since $\langle q^{\mu_i^- - \mu_i^+}/t \rangle =\langle q^{- \mu_i^+}/t \rangle$,
\begin{equation}
(\ref{eq:check l.h.s}) = \prod_{1 \le i \le m}
\frac{\langle q^{1-\nu_i^+}/t \rangle_{\nu_i^{-}}}{\langle q/t \rangle_{\nu_i^-}}.
\end{equation}
Also, when $\mu_i^- > 0$ , since $\langle q^{-\mu_i^+}/t \rangle
=\langle q^{-M}/t \rangle$,
\begin{equation}
(\ref{eq:check l.h.s}) = \prod_{1 \le i \le m}
\frac{\langle q^{\mu_i^-} \rangle
\langle q^{1-M}/t \rangle_{\mu_i^{-}-1}
\langle q^{1-\nu_i^+}/t \rangle_{\nu_i^{-}}}
{\langle q^{1-\nu_i^+}/t \rangle_{\mu_i^-}
\langle q \rangle_{\nu_i^-}}.
\end{equation}
Therefore the point $t=q^{-M}$ in the left-hand side is an apparent singularity.
Note that unless $\mu_i^- =0$ or $\mu_i^+=M$ for each $1 \le i \le m$,
the corresponding term in the left-hand side of (\ref{eq:special of BC duality}) is zero.
We can also check that the point $t=q^{-M}$ in the right-hand side is an apparent singularity.

From the argument above, 
specializing Theorem \ref{Theo:trans form BC} as
$\alpha =(M^m), \beta=(1^n)$ and $t=q^{-M}\ (M=0, 1, 2, \ldots)$,
we find that the factor involving both variables $x$ and $y$ 
in each side of (\ref{eq:special of BC duality}) simplifies
to the form
\begin{align}
\prod_{\substack{i: \mu_i^- >0}}
\frac{T_{q, x_i}^{\mu_i^-} \Psi(x_i;y)}
{\Psi(x_i;y)}
\prod_{\substack{i: \mu_i^+ <M}}
\frac{T_{q, x_i}^{-(M-\mu_i^+)} \Psi(x_i;y)}
{\Psi(x_i;y)}
= \frac{T_{q, x}^{\nu} \Psi(x; y)}{\Psi(x;y)},
\\
\prod_{\substack{k : \lambda_k^+ =0 }}
\frac{T_{t, y_k}^{1- \lambda_k^+} \Psi(x;y_k)}
{\Psi(x;y_k)}
\prod_{\substack{k : \lambda_k^- =1}}
\frac{T_{t, y_k}^{-\lambda_k^-}\Psi(x;y_k)}
{\Psi(x;y_k)}
= \frac{T_{t, y}^{\kappa}\Psi(x;y)}{\Psi(x;y)},
\end{align}
respectively.
Here we set $\nu$ and $\kappa$ as follows:
\begin{align}
\nu_i=
\begin{cases}
\mu_i^- \quad &(\mu_i^- >0, \mu_i^+=M) ,  \\
-(M-\mu_i^+) \quad &(\mu_i^+ < M, \mu_i^-=0) ,\\
0 \quad &(\mbox{otherwise}),
\end{cases} \qquad
\kappa_k=
\begin{cases}
1 \quad &(\lambda_k^+ = 0, \lambda_k^-=0) ,  \\
-1 \quad &(\lambda_k^+ =1, \lambda_k^-=1) ,\\
0 \quad &(\mbox{otherwise}).
\end{cases} 
\end{align}
In this way, 
the left-hand side can be interpreted as the action of a $q$-difference operator 
on the kernel function of dual Cauchy type.
The right-hand side is also expressed by the action of a $t$-difference operator, and
in fact is equal to
\begin{equation}
\Psi(x;y)^{-1}e(u; \sqrt{tq}/a_0)_{mM-n} 
 \mathcal{D}^y(u;a_1, a_2, a_3, a_4|t, q) \Psi(x; y) . 
\end{equation}
We replace the parameters $(a_1, a_2, a_3, a_4)$ with $(a, b, c, d)$.
Then the left-hand side of (\ref{eq:special of BC duality}) can be expressed as
\begin{align}
\Psi(x;y)^{-1} \left( \sum_{l=0}^{mM} e(u;\sqrt{tq}/\alpha)_{mM-l} 
K_l^x \right) \Psi(x;y) 
\end{align}
for some $q$-difference operators $K_l^x$ for which we will determine
the explicit formulas later.
Hence we have 
\begin{align}
\Psi(x;y)^{-1} \left( \sum_{l=0}^{mM} e(u;\sqrt{tq}/\alpha)_{mM-l} 
K_l^x \right) \Psi(x;y)
=\Psi(x;y)^{-1} e(u; \sqrt{tq}/\alpha)_{mM-n}  \widehat{\mathcal{D}}^y(u) \Psi(x;y).
 \label{eq:special case of dual Cauchy}
\end{align}
Comparing (\ref{eq:Theo dual Cauchy}) 
with (\ref{eq:special case of dual Cauchy}),
we obtain that
\begin{align}
&\Psi(x;y)^{-1}
e(u; t^m \widehat{\alpha})_{n} e(u; \sqrt{tq}/\alpha)_{mM-n}
\mathcal{H}^x(u) \Psi(x;y) \nonumber \\
&=
\Psi(x;y)^{-1}\left( \sum_{l=0}^{mM} e(u;\sqrt{tq}/\alpha)_{mM-l} 
K_l^x \right) \Psi(x;y)\nonumber \\
&=\prod_{1 \le i \le m} \frac{\langle tq x_i/a, tq x_i/b, tq x_i/c, tq x_i/d \rangle_{M}}
{\langle tq x_i^2, t^2q x_i^2\rangle_{M}}
\prod_{1 \le i < j \le m}
\frac{\langle t q^{M+1}x_i x_j,  t^2 q^{M+1}x_i x_j\rangle_{M}} 
{\langle tqx_i x_j,  t^2qx_i x_j\rangle_{M}} \nonumber \\
&\cdot
\sum_{0 \le \mu^{-} \le \nu^{-} \le \nu^{+} \le \mu^{+} \le (M^m)}
\Biggl((-1)^{mM+|\nu^+| + |\nu^-|}
e(u; \sqrt{tq}/\alpha)_{|\nu^+|-|\nu^-|} \nonumber \\
&\cdot \prod_{1 \le i \le m} 
\frac{\langle a x_i, b x_i, c x_i, d x_i\rangle_{\nu_i^{-}}}
{\langle  tq x_i/a, tq x_i/b, tq x_i/c, tq x_i/d \rangle_{\nu_i^{+}}} 
\prod_{1 \le i \le j \le m}
\frac{\langle q^{\mu_i^- + \mu_j^-} x_i x_j, t^2 q^{\mu_i^+ + \mu_j^+} x_i x_j\rangle}
{\langle x_i x_j,  t^2 x_i x_j\rangle} \nonumber \\
&\cdot \prod_{1 \le i < j \le m}
\frac{\langle q^{\mu_i^- - \mu_j^-} x_i/x_j,
q^{\mu_i^+ - \mu_j^+} x_i/x_j,
q^{\nu_i^- - \nu_j^-} x_i/x_j,
q^{\nu_i^+ - \nu_j^+} x_i/x_j\rangle
\langle tx_i x_j, qx_i x_j\rangle_{\nu_i^- + \nu_j^-}
}{\langle x_i/x_j, x_i/x_j, x_i/x_j, x_i/x_j\rangle
\langle  tqx_i x_j, t^2 x_i x_j \rangle_{\nu_i^+ + \nu_j^+}} \nonumber \\
&\cdot \prod_{1 \le i, j \le m}
\frac{\langle tq^{\mu_i^- + \mu_j^+} x_i x_j,
tq^{\nu_i^- + \nu_j^+} x_i x_j,
q^{\mu_i^- - \mu_j^+} x_i/tx_j\rangle
\langle  tx_i x_j \rangle_{\mu_i^- + \nu_j^+}
\langle  t^2x_i x_j \rangle_{\mu_i^+ + \nu_j^+}}
{\langle tx_i x_j, tx_i x_j, x_i/tq^{\mu_j^+}x_j\rangle
\langle  qx_i x_j \rangle_{\mu_i^-+ \nu_j^-}
\langle  tqx_i x_j \rangle_{\mu_i^+ + \nu_j^-}}\nonumber \\
&\cdot \prod_{1 \le i, j \le m}
\frac{\langle  x_i x_j, x_i/q^{\nu_j^-}x_j , x_i/tq^{M}x_j \rangle_{\mu_i^-}
\langle  tx_i x_j, x_i/q^{M}x_j  \rangle_{\mu_i^+}
\langle  qx_i/tq^{\nu_j^+}x_j , x_i/q^{\nu_j^+}x_j \rangle_{\nu_i^-}
}
{\langle  tq^{M+1}x_i x_j, qx_i/tq^{\nu_j^+}x_j, qx_i/ x_j \rangle_{\mu_i^-}
\langle  t^2 q^{M+1} x_i x_j, qx_i/ x_j  \rangle_{\mu_i^+}
\langle qx_i/tq^{\mu_j^+}x_j, qx_i/ x_j \rangle_{\nu_i^-}
} 
\nonumber \\
&\cdot 
\prod_{1 \le i, j \le m}
\frac{\langle  x_i/q^{\mu_j^+}x_j \rangle_{\nu_i^+}}
{\langle qx_i/ x_j \rangle_{\nu_i^+}}
\frac{T_{q, x}^{\nu}\Psi(x;y)}{\Psi(x;y)} \Biggr)\label{eq:gene of H_l}.
\end{align}
Since
\begin{equation}
e(u; t^m \widehat{\alpha})_{n} e(u; \sqrt{tq}/\alpha)_{mM-n}
= e(u; \sqrt{tq}/\alpha)_{mM} = H(u; \alpha t^{\rho_m}),
\end{equation}
the left-hand side of (\ref{eq:gene of H_l}) equals
\begin{align}
\Psi(x;y)^{-1}
\left(\sum_{l=0}^{mM} e(u;\sqrt{tq}/\alpha)_{mM-l} H_l^x \right)
\Psi(x;y). 
\end{align}
Hence  we have $H^x_l \Psi(x;y) = K_l^x \Psi(x;y)$ 
for each $l=0, 1,2, \ldots, mM$.
From the formula
(\ref{eq:expand by Koornwinder}) proved by Mimachi,
we obtain
$H^x_l P_\lambda(x) = K_l^x P_\lambda(x)$
for any partition $\lambda \subset (n^m)$.
Since $n$ is the arbitrary non-negative integer,
$H_l^x$ equals $K_l^x$ as a $q$-difference operator.
We see that the row type $q$-difference operators $H_l^x (l \le mM)$
correspond to the terms of the right-hand side of (\ref{eq:gene of H_l}) such that
$|\nu^+|-|\nu^-| =mM-l$.

We compute the explicit formula of $H_{l}^x=K_l^x$ for $l \le M$.
As we will see below their coefficients are expressed as rational functions
in $t^{\frac{1}{2}}$.
Note also that the operator $H_l^x$
does not depend on the non-negative integer $n$.
Since $H_l^x$ are the $W_m$-invariant operators,
it is enough to calculate the coefficients
$H^{(l)}_{\nu}(x):=H^{(l)}_{\nu}(x;a, b,c, d)$ of
$\prod_{1 \le i \le m} T^{\nu_i}_{q, x_i}\ (\nu \in \mathbb{N}^m, 0 \le |\nu| \le l)$.
The coefficients $H^{(l)}_{\nu}(x)$ have the following form:
\begin{align}
&H^{(l)}_{\nu}(x)
=\sum_{\substack{\nu \le \nu^{-} \\ |\nu| \le |\nu^-| \le l}}
\sum_{\substack{\nu^{-} \le \nu^+ \le(M^m) \\ |\nu^+|=mM-l+|\nu^-|}}
A^{(l)}_{\nu, \nu^-, \nu^+}(x). 
\end{align}
Relabeling the indices of summations,
we obtain the explicit formulas $(\ref{eq:explicit of H})$ of 
the coefficients $H^{(l)}_\nu (x)$
for $\nu \in \mathbb{N}^m$ such that $0 \le |\nu| \le l$.

Although we computed the explicit formula of $q$-difference operator $H^x_l$ in the case of $t= q^{-M}$,
for a fixed $l$ this expression with $t=q^{-M}$ is valid for any
$M=l, l+1, \ldots$.
Thus the explicit formula (\ref{eq:explicit of H}) is valid for any parameter $t$
and we complete the proof of Theorem \ref{Theo:explicit of H_l}.
\end{proof}

\subsection{Pieri formulas}
\label{subsec:Pieri}
It is known that
the Koornwinder polynomials have the duality property \cite{Diejen2,Sahi}:
\begin{equation}
\frac{P_\lambda(a t^{\rho_m}q^\mu ; a, b, c, d| q, t)}{P_\lambda(a t^{\rho_m}; a, b, c, d| q, t)}
 = \frac{P_\mu(\alpha t^{\rho_m}q^\lambda ; \alpha, \beta, \gamma, \delta| q, t)}
 {P_\mu(\alpha t^{\rho_m}; \alpha, \beta, \gamma, \delta| q, t)}, \label{eq:duality}
\end{equation}
where the parameters $\alpha, \beta, \gamma, \delta$
are defined by
\begin{equation}
\alpha =\sqrt{abcd/q}, \beta=ab/\alpha,
\gamma=ac/\alpha ,\delta=ad/\alpha.
\end{equation}
Van Diejen derived the Pieri formula of column type from the duality 
of the Koornwinder polynomials and $D_r$.
In this subsection, we present the ``Pieri formula of row type'' by using
the $q$-difference operators $H_l$.
For any partition $\mu=(\mu_1, \ldots, \mu_m)$,
we define $P_{\mu}^+$ by
\begin{equation}
P_{\mu}^+ = \{\lambda \in \mathbb{N}^m| \lambda_1 \ge \cdots 
\ge \lambda_m \ge 0,\ 
\lambda'_j-\mu'_j\in \{ \pm1, 0 \} \ (1 \le j \le m) \},
\end{equation}
where $\lambda'$ stands for the conjugate of a partition $\lambda$.
In other words, $P^+_\mu$
is the set of partitions
obtained by adding or subtracting at most one  to $\mu$ in each column.
By direct calculation, we obtain the following lemma.
\begin{lemm}\label{lemm:sub of pieri}
Let $\mu$ be a partition.
For any multi-index $\nu \in \mathbb{Z}^m$,
if $\mu + \nu= (\mu_1 + \nu_1 ,\ldots, \mu_m + \nu_m) \notin P^+_\mu$ ,
$H^{(l)}_{\nu}(at^{\rho_m}q^{\mu};a,b,c,d)=0$.
\end{lemm}
For any partition $\mu$ , by substituting
$x = a t^{\rho_m}q^\mu $ in (\ref{eq:q-diff eq of H_l}),
we obtain
\begin{align}
\sum_{\substack{\nu \in \mathbb{Z}^m \\ 0 \le |\nu| \le l}}
H^{(l)}_{\nu}(at^{\rho_m}q^{\mu};a,b,c,d)
\frac{P_\lambda(at^{\rho_m}q^{\mu+ \nu})}{P_\lambda(a t^{\rho_m})}
=  h_l(\alpha t^{\rho_m}q^\lambda; \alpha|q, t) 
 \frac{P_\lambda(at^{\rho_m}q^{\mu})}{P_\lambda(a t^{\rho_m})}
. \label{eq:subs of q-diff of H}
\end{align}
From Lemma \ref{lemm:sub of pieri}, we can apply the duality of Koornwinder polynomials 
to (\ref{eq:subs of q-diff of H}) to obtain
\begin{align}
\sum_{\substack{\nu \in \mathbb{Z}^m \\ 0 \le |\nu| \le l
\\ \mu+\nu \in P^+_\mu}}
H^{(l)}_{\nu}(a t^{\rho_m}q^{\mu};a,b,c,d)
\frac{P_{\mu+ \nu}(\alpha t^{\rho_m} q^{\lambda};\alpha, \beta,\gamma, \delta|q,t )}
{P_{\mu+ \nu}(\alpha t^{\rho_m};\alpha, \beta,\gamma, \delta|q,t )} 
\nonumber \\
 =  h_l(\alpha t^{\rho_m} q^\lambda ; \alpha|q, t)
 \frac{P_\mu(\alpha t^{\rho_m} q^{\lambda};\alpha, \beta,\gamma, \delta|q,t )}
 {P_\mu(\alpha t^{\rho_m};\alpha, \beta,\gamma, \delta|q,t )}. 
\end{align}
Replacing $\alpha t^{\rho_m}q^{\lambda}$ and 
the parameters $(\alpha, \beta, \gamma, \delta)$
with $x$ and $(a, b, c, d)$, respectively, 
we obtain the following Pieri formula of row type.
\begin{theo}
For any non-negative integer $l=0, 1, 2, \ldots$, we have the Pieri formula of row type:  
\begin{equation}
h_l (x; a|q, t) \frac{P_\mu (x)}{P_\mu(a t^{\rho_m})} = 
\sum_{\substack{\nu \in \mathbb{Z}^m \\ 0 \le |\nu| \le l \\
\mu+\nu \in P^+_\mu}}
H^{(l)}_{\nu}(\alpha t^{\rho_m} q^{\mu};\alpha, \beta, \gamma, \delta)
\frac{P_{\mu+ \nu}(x)}{P_{\mu + \nu }(a t^{\rho_m})}.
\end{equation}
\end{theo}
\begin{center}
\textbf{Acknowledgments}
\end{center}
The author would like to express his thanks to 
Professors Masatoshi Noumi and Yasushi Komori for various advices.

\end{document}